\documentclass{amsart}
\usepackage{amssymb,amsthm}
\usepackage{amsmath}
\usepackage{setspace}
\usepackage{dcpic,pictexwd}
\usepackage[all]{xy}
\usepackage[pdftex]{graphicx}
\usepackage[pdftex]{hyperref}
\usepackage{bbm}

\theoremstyle{definition}
 \newtheorem{definition}{Definition}[section]

\theoremstyle{plain}
 \newtheorem{proposition}[definition]{Proposition}

\theoremstyle{plain}
 \newtheorem{theorem}[definition]{Theorem}
 \newtheorem{claim}[definition]{Claim}
\theoremstyle{definition}

\theoremstyle{plain}

\theoremstyle{plain}
 \newtheorem{corollary}[definition]{Corollary}

\theoremstyle{remark}
 \newtheorem{remark}[definition]{Remark}

\theoremstyle{definition}

\theoremstyle{plain}

\newcommand{\Ext}{\mathrm{Ext}}
\newcommand{\End}{\mathrm{End}}

\newcommand{\Mat}{\mathrm{Mat}}
\newcommand{\Hom}{\mathrm{Hom}}
\newcommand{\SHom}{\underline{\mathrm{Hom}}}

\newcommand{\Ca}{\mathcal{C}}
\newcommand{\Fun}{\mathrm{F}}

\newcommand{\Def}{\mathrm{Def}}
\newcommand{\Sets}{\mathrm{Sets}}

\newcommand{\Z}{\mathbb{Z}}
\newcommand{\SEnd}{\underline{\End}}
\newcommand{\A}{\Lambda}
\newcommand{\surjection}{\twoheadrightarrow}
\newcommand{\injection}{\hookrightarrow}

\renewcommand{\k}{\Bbbk}
\renewcommand{\1}{\mathbbm{1}}
\newcommand{\invlim}{\varprojlim}

\hypersetup{
    bookmarks=true,         
    unicode=false,          
    pdftoolbar=true,        
    pdfmenubar=true,        
    pdffitwindow=false,     
    pdfstartview={FitH},    
    pdftitle={Universal Deformation Rings for Complexes Over Finite-Dimensional Algebras},    
    pdfauthor={Jose A. Velez-Marulanda},     
    pdfsubject={Derived Categories, Singularity Categories, Universal deformation rings},   
    pdfcreator={Jose A. Velez-Marulanda},   
    pdfproducer={Jose A. Velez-Marulanda}, 
    pdfkeywords={Derived Categories} {Singularity Categories}{Universal Deformation Rings}, 
    pdfnewwindow=true,      
    colorlinks=true,       
    linkcolor=red,          
    citecolor=blue,        
    filecolor=magenta,      
    urlcolor=blue           
}

\title[Universal Deformation Rings of String Modules]{Universal deformation rings of string modules over certain class of self-injective special biserial algebras}
\author[Calder\'on-Henao]{Yohny Calder\'on-Henao}
\address{Instituto de Matem\'aticas, Universidad de Antioquia, Medell{\'\i}n, Antioquia, Colombia}
\email{yohny.calderon@udea.edu.co}
\author[Giraldo]{Hern\'an Giraldo}
\address{Instituto de Matem\'aticas, Universidad de Antioquia, Medell{\'\i}n, Antioquia, Colombia}
\email{hernan.giraldo@udea.edu.co}
\author[Rueda-Robayo]{Ricardo Rueda-Robayo}
\address{Instituto de Matem\'aticas, Universidad de Antioquia, Medell{\'\i}n, Antioquia, Colombia}
\email{ricardo.rueda@udea.edu.co}
\author[V\'elez-Marulanda]{Jos\'e A. V\'elez-Marulanda}
\address{Department of Mathematics, Valdosta State University,
2072 Nevins Hall, 1500 N. Patterson St, Valdosta, GA,  31698-0040}
\email{javelezmarulanda@valdosta.edu}

\thanks{The fourth author was supported by the Release Time for Research Scholarship of the Office of Academic Affairs and by the Faculty Research Seed Grant of the Office of Sponsored Programs \& Research Administration at the Valdosta State University. This research was also partly supported by CODI and Estrategia de Sostenibilidad 2016-2017 (Universidad de Antioquia), and COLCIENCIAS-ECOPETROL (Contrato RC. No. 0266-2013).}

\begin{document}
\renewcommand{\labelenumi}{\textup{(\roman{enumi})}}
\renewcommand{\labelenumii}{\textup{(\roman{enumi}.\alph{enumii})}}
\numberwithin{equation}{section}

\begin{abstract}
Let $\k$ be an algebraically closed field of arbitrary characteristic, let $\A$ be a finite dimensional $\k$-algebra and let $V$ be a $\A$-module with stable endomorphism ring isomorphic to $\k$. If $\A$ is self-injective, then $V$ has a universal deformation ring $R(\A,V)$, which is a complete local commutative Noetherian $\k$-algebra with residue field $\k$. Moreover,  if $\Lambda$ 
is further a Frobenius $\k$-algebra, then $R(\Lambda,V)$ is stable under syzygies.  We use these facts to determine the universal deformation rings of string $\Lambda_N$-modules whose corresponding stable endomorphism ring is isomorphic to $\k$, and which lie either in a connected component of the stable Auslander-Reiten quiver of $\A_{m,N}$ containing a module with endomorphism ring isomorphic to $\k$ or in a periodic component containing only string $\A_{m,N}$-modules, where $m\geq 3$ and $N\geq 1$ are integers, and $\Lambda_{m,N}$ is a self-injective special biserial $\k$-algebra.
\end{abstract}
\maketitle

\section{Introduction}\label{int}

Let $\k$ be a field of arbitrary characteristic, and denote by $\hat{\Ca}$ the category of all complete local commutative Noetherian $\k$-algebras with residue field $\k$. In particular, the morphisms in $\hat{\Ca}$ are continuous $\k$-algebra homomorphisms that induce the identity on $\k$. In this article, all modules are assumed to be from the left side and finitely generated. Suppose 
that $\A$ is a fixed finite dimensional $\k$-algebra and let $V$ be a $\A$-module.  We denote by $\End_\A(V)$ (resp., by $\SEnd_\A(V)$) the endomorphism ring (resp., the stable endomorphism ring) of $V$ (see e.g. \cite[\S IV.1]{auslander}). We denote by $\Gamma_s(\A)$ the stable Auslander-Reiten quiver of $\A$ (see \cite[VII]{auslander} and e.g. \cite[\S I.8.2]{erdmann}). Let $R$ be an arbitrary object in $\hat{\Ca}$. Following \cite{blehervelez}, a {\it lift} $(M,\phi)$ of $V$ over $R$ is a $R\otimes_\k\A$-module $M$ that is free over $R$ together with an isomorphism of $\A$-modules $\phi:\k\otimes_RM\to V$. If $\A$ is self-injective and the stable endomorphism ring of $V$ is isomorphic to $\k$, 
then there exists a particular object  $R(\A,V)$ in $\hat{\Ca}$ and a lift $(U(\A,V),\phi_{U(\A,V)})$ of $V$ over $R(\A,V)$, which is universal with respect to all isomorphism classes of 
lifts of $V$ over such $\k$-algebras $R$ (see \cite{blehervelez} and \S \ref{section2}). The ring $R(\A,V)$ and the isomorphism class of the lift $(U(\A,V),\phi_{U(\A,V)})$ are 
respectively called the {\it universal deformation ring} and the {\it universal deformation} of $V$.  In \cite{bleher9,blehervelez,velez}, universal deformation rings of modules over certain self-injective algebras, which are not Morita equivalent to a block of a group algebra, are also discussed. More recently, it was proved in \cite[Prop. 3.2.5]{blehervelez2} that the isomorphism class of universal deformation rings of modules is an invariant under stable equivalences of Morita type (as introduced by M. Brou\'e in \cite{broue}) between self-injective $\k$-algebras. These results together with \cite[Thm. 3.4]{holm} were used in \cite[\S 4]{blehervelez} to classify (up to stable equivalence of Morita type) the isomorphism class of universal deformation rings of modules whose endomorphism ring is isomorphic to $\k$ over the class of algebras of dihedral type $D(3\mathcal{R})$ (see \cite[\S VI.5]{erdmann}). 

Traditionally, universal deformation rings are studied when $\A$ is equal to a group algebra $\k G$, where $G$ is a finite group and $\k$ has positive characteristic $p$ (see e.g., 
\cite{bleher1,bleher7,bleher2,bleher3,bleher4,bleher5,bleher6}).  
Deformation of modules over more general and other types of finite dimensional algebras have been studied by various authors in different contexts (see e.g., \cite{auslander4,ile,yau} and their references).  We refer the reader to \cite{blehervelez2} in order to obtain information about deformation of complexes for finite dimensional algebras in the context of derived categories and derived equivalence. 

Assume that $\k$ is algebraically closed, let $m$ and $N$ be integers with $m\geq 3$ and $N\geq 1$, and let $\A_{m,N}$ the basic $\k$-algebra $\A_{m,N}$ as in Figure \ref{fig1}. 
\begin{figure}[ht]
\begin{align*}
Q&=\xymatrix@1@=25pt{
&\underset{0}{\bullet}\ar@/^/[dl]^{\bar{a}_{m-1}}\ar@/^/[rr]^{a_0}&&\underset{1}{\bullet}\ar@/^/[ll]^{\bar{a}_0}\ar@/^/[dr]^{a_1}&\\
\underset{m-1}{\bullet}\ar@/^/[ur]^{a_{m-1}}\ar@/^/[d]^{\bar{a}_{m-2}}&&&	&\underset{2}{\bullet}\ar@/^/[ul]^{\bar{a}_1}\ar@/^/[d]^{a_2}\\
\underset{m-2}{\bullet}\ar@/^/[u]^{a_{m-2}}\ar@/^/@{.>}[dr]^{}&&&	&\underset{3}{\bullet}\ar@/^/[u]^{\bar{a}_2}\ar@/^/@{.>}[dl]^{}\\
&\underset{\ast}{\bullet}\ar@/^/@{.>}[ul]^{}\ar@/^/@{.>}[rr]^{}&&\underset{\ast}{\bullet}\ar@/^/@{.>}[ll]^{}\ar@/^/@{.>}[ur]^{}&
}\\\\
I_{m,N}&= \langle a_{i+1}a_i, \bar{a}_{i-1}\bar{a}_i, (\bar{a}_ia_i)^N-(a_{i-1}\bar{a}_{i-1})^N :i\in \Z/m\rangle.
\end{align*}
\caption{The basic $\k$-algebra $\A_{m,N}=\k Q/I_{m,N}$.}\label{fig1}
\end{figure}


The $\k$-algebra $\A_{m,N}$ is a self-injective (so Frobenius by e.g. \cite[Cor. 4.3]{skow}) special biserial $\k$-algebra (as introduced in \cite{wald})), it follows that all the non-projective indecomposable $\A_{m,N}$-modules can be described combinatorially by using so-called string and bands for $\A_{m,N}$; the corresponding indecomposable $\A_{m,N}$-modules are called string and band modules (see \cite{buri}). In this article, we are interested only in these string $\A_{m,N}$-modules.  

These $\k$-algebras $\A_{m,N}$ have been studied by various authors under different contexts (see e.g. \cite{erdmann2,snashall2,snashall1} and their references). In particular, if $m=3$ and $N=1$, then $\A_{3,1}=D(3\mathcal{K})^{1,1,1}$ is an algebra of dihedral type of polynomial growth as discussed by K. Erdmann and A. Skowro\'nski in  \cite[\S 4]{erdmann3}. Moreover, it follows from \cite[\S V.2.4.1]{erdmann} that if $\mathrm{char}\,\k=2$, then $\A_{3,1}$ is isomorphic to the group algebra of the alternating group $A_4$. On the other hand $\A_{3,1}$ is among the ones studied by F. M. Bleher in \cite[\S 3.2, \S 4.2]{bleher1} (for the case $d=2$ with $\mathrm{char}\,\k=2$), and with S. N. Talbott in \cite{bleher9}.  In the latter, which uses the results in \cite{blehervelez}, they proved that if $\mathfrak{C}$ is a non-periodic component of $\Gamma_s(\A_{3,1})$, then every module $V$ belonging to $\mathfrak{C}$ has stable endomorphism ring $\SEnd_{\A_{3,1}}(V)$ isomorphic to $\k$, and that the universal deformation ring $R(\A_{3,1},V)$ of $V$ is also isomorphic to $\k$ (see \cite[Prop. 3.1(iii)]{bleher9}). If $\mathfrak{T}$ is a $3$-tube of $\Gamma_s(\A_{3,1})$, then there is precisely one $\Omega$-orbit of $\A_{3,1}$-modules $W$ in $\mathfrak{T}$ whose stable endomorphism ring $\SEnd_{\A_{3,1}}(W)$ is isomorphic to $\k$, and $R(\A_{3,1},W)$ is also isomorphic to $\k$ (see \cite[Prop. 3.2(iii)]{bleher9}).  

The reader is invited to look at the references in e.g. \cite{erdmann2,snashall2,snashall1} for getting further information concerning the algebras $\A_{m,N}$ in different settings. 

The main goal of this article is to prove the following result (for more specific results and details see Remark \ref{rem5}, Corollary \ref{cor3.4}, and Propositions  \ref{prop4}, \ref{prop5}, \ref{prop7} and \ref{prop6}). 

\begin{theorem}\label{thm1}
Let $\A_{m,N}$ be the basic $\k$-algebra as in Figure \ref{fig1}, and let $\Gamma_s(\A_{m,N})$ denote the stable Auslander-Reiten quiver of $\A_{m,N}$.
\begin{enumerate}
\item If $N=1$, then there are at most finitely many components $\mathfrak{C}$ of $\Gamma_s(\A_{m,N})$ containing a string $\A_{m,N}$-module of minimal string length whose endomorphism ring is isomorphic to $\k$.
\item Assume that $N\geq 2$.
\begin{enumerate}
\item If $m$ is odd, then there are at most finitely many components $\mathfrak{C}$ of type $\mathbb{ZA}_\infty^\infty$ of $\Gamma_s(\A_{m,N})$ containing a string $\A_{m,N}$-module $V$ with $\End_{\A_{m,N}}(V)\cong \k$.
 \item If $m$ is even, then there are infinitely many components $\mathfrak{C}$ of type $\mathbb{ZA}_\infty^\infty$ of $\Gamma_s(\A_{m,N})$ containing a string $\A_{m,N}$-module $V$ with $\End_{\A_{m,N}}(V)\cong \k$.
\end{enumerate}
\item Let $\mathfrak{C}$ be a component of $\Gamma_s(\A_{m,N})$ containing a string $\A_{m,N}$-module $V$ whose endomorphism ring is isomorphic to $\k$, and let $W$ be a $\A_{m,N}$-module lying in $\mathfrak{C}$ with $\SEnd_{\A_{m,N}}(W)\cong \k$.

\begin{enumerate}
\item If $N=1$, then the universal deformation ring $R(\A_{m,1},W)$ is isomorphic to $\k$. 
\item If $N\geq 2$, then the universal deformation ring $R(\A_{m,N},W)$ is isomorphic either to $\k$, or to $\k[[t]]/(t^N)$, or to $\k[[t]]$. 
\end{enumerate}
\item Let $\mathfrak{T}$ be a tube of $\Gamma_s(\A_{m,N})$ containing only string $\A_{m,N}$-modules, and let $W$ be a $\A_{m,N}$-module lying in $\mathfrak{T}$ with $\SEnd_{\A_{m,N}}(W)\cong \k$. Then $R(\A_{m,N},W)$ is isomorphic either to $\k$ or to $\k[[t]]$. 
\end{enumerate}
\end{theorem}

In order to prove Theorem \ref{thm1}, we use similar methods to those in \cite{bleher9,blehervelez} and \cite{velez}. Note that Theorem \ref{thm1} provides an extension of the results obtained in \cite{bleher9} for the algebra $\A_{3,1}$. Theorem \ref{thm1} also shows the remarkable difference between the cases $m$ odd and $m$ even for the $\k$-algebra $\A_{m,N}$ concerning the components of $\Gamma_s(\A_{m,N})$ with $N\geq 2$.  

This article is organized as follows. In \S \ref{section2}, we recall the definitions of deformations and universal deformation rings and summarize some of their properties. In \S \ref{section3}, we describe the radical series of the indecomposable projective modules for $\A_{m,N}$, and classify all $\A_{m,N}$-modules whose endomorphisms rings are isomorphic to $\k$ (see Proposition \ref{prop3.1}). In \S \ref{section4}, we prove Theorem \ref{thm1}.  For the convenience of the reader, in \S \ref{appendix} we describe some basic aspects concerning the representation theory of $\A_{m,N}$, including a precise description of string modules for $\A_{m,N}$ and of the corresponding components of $\Gamma_s(\A_{m,N})$  using hooks and co-hooks (see \cite{buri}). Moreover, we also give a description of the homomorphisms between strings modules as determined by H. Krause in \cite{krause}. 

We refer the reader to e.g., \cite{assem3,auslander,benson,erdmann} for getting further information about basic concepts from the representation theory of finite dimensional algebras such as the definition and properties of the syzygy functor $\Omega$ as well as the definition of the stable Auslander-Reiten quiver of an arbitrary Artinian algebra $\A$.

\section{Preliminary Results on Deformations and Universal Deformation Rings}
\label{section2}
Let $\k$ be a field of arbitrary characteristic and denote by $\hat{\Ca}$ the category of all complete local commutative Noetherian $\k$-algebras with residue field $\k$. In particular, the morphisms in $\hat{\Ca}$ are continuous $\k$-algebra homomorphisms that induce the identity on $\k$. Let $\A$ be a fixed finite dimensional $\k$-algebra. Recall that $\A$ is said to be self-injective if the regular left $\A$-module $_\A\A$ is injective, and that $\A$ is called a Frobenius algebra provided that the left $\A$-modules ${_\A}\A$ and $(\A_\A)^\ast=\Hom_\k(\A_\A,\k)$ are isomorphic.  By \cite
[Prop. 9.9]{curtis}, every Frobenius algebra is self-injective. It is also well-known that a basic algebra is self-injective if and only if it is Frobenius (see e.g. \cite[Cor. 4.3]{skow}). Let $V$ be a $\A$-module that has finite dimension over $\k$. We denote by $\End_\A(V)$ (resp., by $\SEnd_\A(V)$) the endomorphism ring (resp., the stable endomorphism ring) of $V$ (see e.g. \cite[\S IV.1]{auslander}), and we denote the first syzygy of  $V$ by $\Omega V$, i.e., $\Omega V$ is the kernel of a projective cover $P_V\to V$ (see e.g., \cite[pp. 124-126]{auslander}).
Following \cite{blehervelez}, a {\it lift} $(M,\phi)$ of $V$ over $R$ is a finitely generated $R\otimes_\k\A$-module $M$ that is free over $R$ together with an isomorphism of $\A$-modules $\phi:\k\otimes_RM\to V$. Two lifts $(M,\phi)$ and $(M',\phi')$ over $R$ are said to be {\it isomorphic} provided that there exists an $R\otimes_\k\A$-module 
isomorphism $f:M\to M'$ such that $\phi'\circ (\text{id}_\k\otimes f)=\phi$, where $\text{id}_\k$ denotes the identity map on $\k$. If $(M,\phi)$ is a lift of $V$ over $R$, we denote by $[M,\phi]$ its isomorphism class and say that $[M,\phi]$ is a {\it deformation} of $V$ over $R$. We denote by $\Def_\A(V,R)$ the 
set of all deformations of $V$ over $R$. The {\it deformation functor } over $V$ is the 
covariant functor $\hat{\Fun}_V:\hat{\Ca}\to \Sets$ defined as follows: for all objects $R$ in $\hat{\Ca}$, define $\hat{\Fun}_V(R)=\Def_\A(V,R)$; and for all morphisms $\alpha:R\to 
R'$ in 
$\hat{\Ca}$, 
let $\hat{\Fun}_V(\alpha):\Def_\A(V,R)\to \Def_\A(V,R')$ be defined as $\hat{\Fun}_V(\alpha)([M,\phi])=[R'\otimes_{R,\alpha}M,\phi_\alpha]$, where $\phi_\alpha: \k\otimes_{R'}
(R'\otimes_{R,\alpha}M)\to V$ is the composition of $\A$-module isomorphisms 
\[\k\otimes_{R'}(R'\otimes_{R,\alpha}M)\cong \k\otimes_RM\xrightarrow{\phi} V.\]  

Following \cite[\S 2.6]{sch}, we call the set $\hat{\Fun}_V(\k[[t]]/(t^2))$ the tangent space of $\hat{\Fun}_V$, which has an structure of a $\k$-vector space by \cite[Lemma 2.10]{sch}. 
By using Schlessinger's criteria \cite[Thm. 2.11]{sch} and using methods similar to those in \cite{mazur}, it was proved in \cite[Prop. 2.1]{blehervelez} that the deformation functor $\hat{\Fun}_V$ is continuous (see \cite[\S VI.14]{mazur}), that there exists an isomorphism of $\k$-vector spaces 
\begin{equation}\label{hoch}
\hat{\Fun}_V(\k[[t]]/(t^2))\to \Ext_\A^1(V,V),
\end{equation}
and that there exists an object $R(\A,V)$ in $\hat{\Ca}$ and a deformation $[U(\A,V), \phi_{U(\A,V)}]$ of $V$ over $R(\A,V)$ with the following property. For each object $R$ in $\hat{\Ca}$ 
and for all lifts $M$ of $V$ over $R$, there exists a morphism $\upsilon:R(\A,V)\to R$ in $\hat{\Ca}$ such that 
\[\hat{\Fun}_V(\upsilon)[U(\A,V), \phi_{U(\A,V)}]=[M,\phi],\]
and moreover, $\upsilon$ is unique if $R$ is the ring of dual numbers $\k[[t]]/(t^2)$.  In this situation $R(\A,V)$ and $[U(\A,V), \phi_{U(\A,V)}]$ are called respectively the {\it versal deformation ring} and {\it versal deformation} of $V$. If $R(\A,V)$ represents $\hat{\Fun}_V$ in the sense that the functors $\hat{\Fun}_V(-)$ and $\Hom_{\hat{\Ca}}(R(\A,V),-)$ are isomorphic, then we call $R(\A,V)$ and $[U(\A,V), \phi_{U(\A,V)}]$ respectively the {\it universal deformation ring} and {\it universal deformation} of $V$. 
Assume that $V$ has an universal deformation ring $R(\A,V)$. From the isomorphism of $\k$-vector spaces (\ref{hoch}), it follows that if $\Ext_\A^1(V,V)=0$, then $R(\A,V)\cong \k$, and if $\Ext_\A^1(V,V)=\k$, then $R(\A,V)$ is isomorphic to a quotient of $\k[[t]]$. 
On the other hand, in \cite[Prop. 2.5]{blehervelez} it was proved that the isomorphism class of versal deformation rings of modules is preserved under equivalences of Morita type. The following result that concerns modules over self-injective algebras is one of the main results in \cite{blehervelez}.

\begin{theorem}{(\cite[Thm. 2.6]{blehervelez})}\label{thm3}
Let $\A$ be a finite dimensional self-injective $\k$-algebra, and suppose that $V$ is a finitely generated $\A$-module whose stable endomorphism ring $\SEnd_\A(V)$ is isomorphic 
to $\k$.
\begin{itemize}
\item[\textup{(i)}] The module $V$ has a universal deformation ring $R(\A,V)$.
\item[\textup{(ii)}] If $P$ is a finitely generated projective $\A$-module, then $\SEnd_\A(V\oplus P)\cong \k$ and $R(\A,V\oplus P)\cong R(\A,V)$.
\item[\textup{(iii)}] If $\A$ is also a Frobenius algebra, then $\SEnd_\A(\Omega V)\cong \k$ and $R(\A,\Omega V)\cong R(\A,V)$. 
\end{itemize}
\end{theorem}

More recently, it was proved in \cite[Prop. 3.2.5]{blehervelez2} that the isomorphism class of versal deformation rings is preserved under stable equivalences of Morita type (as introduced by M. Brou\'e in\ \cite{broue}) between self-injective $\k$-algebras.

\begin{remark}\label{rem1}
It is proved in Claim 3 within the proof of \cite[Thm. 2.6]{blehervelez} that if $\A$ is self-injective and $(M,\phi)$ is a lift of $V$ 
over an object $R$ in $\hat{\Ca}$ with $\SEnd_\A(V)\cong \k$, then the deformation $[M,\phi]$ corresponding to the lift $(M, \phi)$ does not depend on the particular choice of the $\A$-module isomorphism. More 
precisely, if $f:M\to M'$ is an $R\otimes_\k \A$-module isomorphism with $(M',\phi')$ a lift of $V$ over $R$, then there exists an $R\otimes_\k\A$-module isomorphism $\bar{f}:M\to M'$ 
such that $\phi'\circ (\mathrm{id}_\k\otimes _R\bar{f})=\phi$, i.e., $[M,\phi]=[M',\phi']$ in $\hat{\Fun}_V(R)=\Def_\A(V,R)$. 
\end{remark}

\section{Classification of the string $\A_{m,N}$-modules with endomorphism ring isomorphic to $\k$}\label{section3}
For the remainder of this article, let $\k$ be a fixed algebraically closed field of arbitrary characteristic. Let $m\geq 3$ and $N\geq 1$ be fixed but arbitrary integers, and let $\A_{m,N}=\k Q/I_{m,N}$ be as in Figure \ref{fig1}. We identify the vertices of $Q$ with elements of $\mathbb{Z}/m$ (the cyclic group of $m$ elements). Recall that these algebras are all basic self-injective (so Frobenius). A brief description of the representation theory of $\A_{m,N}$ is provided in  \S\ref{appendix}. Recall that $\Gamma_s(\A_{m,N})$ denotes the stable Auslander-Reiten quiver of $\A_{m,N}$.

\begin{remark}\label{rem0}
For all $m\geq 3$ and $N\geq 1$, the $\k$-algebra $\A_{m,N}$ is the Brauer graph algebra of a polygon with $m$ edges and multiplicity $N$ at every vertex. This implies in particular that the $\k$-algebras $\A_{m,N}$ are all symmetric (see \cite[Thm. 1.1]{schroll}). Therefore, it follows from \cite[Prop. IV.3.8]{auslander} that for all $\A_{m,N}$-modules $V$, $\tau_{\A_{m,N}} V=\Omega^2V$, where $\tau_{\A_{m,N}}$ denotes the Auslander translation of $\Gamma_s(\A_{m,N})$ (see \cite[Prop. IV.3.8 \& \S VII.1]{auslander}). In particular, for all connected components $\mathfrak{C}$ of $\Gamma_s(\A_{m,N})$, we have $\Omega^2\mathfrak{C}=\mathfrak{C}$. 
\end{remark}

For all vertex $i\in \Z/m$ of $Q$, the radical series of the projective indecomposable $\A_{m,N}$-module $P_i$ can be described as in Figure \ref{projec}.
\begin{figure}[ht]
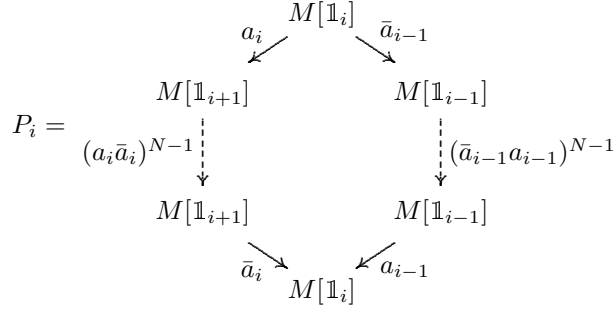

\begin{align*}
P_i=\begindc{\commdiag}[150]
\obj(-3,3)[v1]{$M[\1_i]$}
\obj(-6,1)[v2]{$M[\1_{i+1}]$}
\obj(-6,-2)[v3]{$M[\1_{i+1}]$}
\obj(-3,-4)[v4]{$M[\1_i]$}
\obj(0,1)[v6]{$M[\1_{i-1}]$}
\obj(0,-2)[v7]{$M[\1_{i-1}]$}
\mor{v1}{v2}{$a_i$}[-1,0]
\mor{v2}{v3}{$(a_i\bar{a}_i)^{N-1}$}[-1,1]
\mor{v3}{v4}{$\bar{a}_i$}[-1,0]
\mor{v1}{v6}{$\bar{a}_{i-1}$}[1,0]
\mor{v6}{v7}{$(\bar{a}_{i-1}a_{i-1})^{N-1}$}[1,1]
\mor{v7}{v4}{$a_{i-1}$}[1,0]
\enddc
\end{align*}
\caption{The radical series of the indecomposable projective $\A_{m,N}$-module corresponding to the vertex $i\in \Z/m$.}\label{projec}
\end{figure} 
\subsection{$\A_{m,N}$-modules with endomorphism ring isomorphic to $\k$}

For all strings $C$ for $\A_{m,N}$, we read $C$ from right to left, and denote by $\mathbf{s}(C)$ and $\mathbf{t}(C)$, the vertex where $C$ starts and ends, respectively. 
Let $i\in \Z/m$ be fixed but arbitrary. Let $Z_i'$ be the string representative for $\A_{m,N}$ 
\begin{equation*}\label{zigzag1}
Z_i' = \cdots a_{i+2}\bar{a}_{i+1}^{-1}a_i
\end{equation*}
such that  
\begin{equation*}
\mathbf{t}(Z_i')=
\begin{cases}
i-1, & \text{ if $m$ is odd,}\\
i, & \text{ otherwise,}
\end{cases}
\end{equation*}
and $Z'_i$ has substrings equivalent neither to $(a_k\bar{a}_k)^l$ nor $(\bar{a}_ka_k)^l$ for some $k\in \Z/m$ and $l\geq 1$. 
Similarly, let $Z_i''$ be the string representative for $\A_{m,N}$ 
\begin{equation*}\label{zigzag2}
Z_i'' = \cdots \bar{a}_{i+2}^{-1}a_{i+1}\bar{a}_i^{-1}
\end{equation*}
such that  
\begin{equation*}
\mathbf{t}(Z_i'')=
\begin{cases}
i-1, & \text{ if $m$ is odd,}\\
i, & \text{ otherwise,}
\end{cases}
\end{equation*}
and $Z''_i$ has substrings equivalent neither to $(a_k\bar{a}_k)^l$ nor $(\bar{a}_ka_k)^l$ for some $k\in \Z/m$ and $l\geq 1$.
For all $i\in \Z/m$, we let  
\begin{equation}\label{zigzag}
Z_i\in \{Z_i',Z_i''\}. 
\end{equation}

For example, if $m=6$, then 
\begin{align*}
Z_0'&=\bar{a}_5^{-1}a_4\bar{a}_3^{-1}a_2\bar{a}_1^{-1}a_0 &&\text{ and }&& Z_0''=a_5\bar{a}_4^{-1}a_3\bar{a}_2^{-1}a_1\bar{a}_0^{-1}.
\end{align*}

Note that that for all $i\in \Z/m$, the string length of $Z_i$ is $m-1$ when $m$ is odd, and $m$ otherwise. 
If $m$ is even, then for all $l\geq 0$ we define {\bf strings} representatives $Z_i^{(l)}$ for $\A_{m,N}$ inductively as follows:
\begin{align}\label{induczigzag}
Z_i^{(0)}&=Z_i, &&\text{ and }\\
Z_i^{(l+1)}&=\theta_{\mathbf{t}\left(Z_i^{(l)}\right)}Z_i^{(l)} &&\text{for all $l\geq 0$},\notag
\end{align}
where for all $k\in \Z/m$, $\theta_k\in \{\bar{a}_{k+1}^{-1}a_k, a_{k+1}\bar{a}_k^{-1}\}$. For all integers $n\geq 0$, we define the string $W_i^{(n)}$ for $\A_{m,N}$ as 
\begin{equation}\label{stringW}
W_i^{(l)}=\begin{cases}
Z_i,&\text{ if $m$ is odd,}\\
Z_i^{(l)},&\text{ otherwise.}
\end{cases}
\end{equation} 
\begin{proposition}\label{prop3.1}
Let $M[S]$ be a string $\A_{m,N}$-module. Then $M[S]$ has endomorphism ring isomorphic to $\k$ if and 
only if there exists $i\in \Z/m$ such that the string representative $S$ is equivalent either to $\1_i$, or to $a_i$, or to $\bar{a}_i$, or to a substring of $Z_i$ as in (\ref{zigzag}), or to $W_i^{(l)}$ as in (\ref{stringW}) for some $l\geq 0$.
\end{proposition}
\begin{proof}
If $S$ is a string representative for $\A_{m,N}$ as in the hypothesis of the ``if" part of Proposition \ref{prop3.1}, then it follows from \S\ref{appendix3} that $\End_{\A_{m,N}}(M[S])=\k$. Assume then that $S$ is a string representative for $\A_{m,N}$ such that $\End_{\A_{m,N}}(M[S])=\k$, and let $n$ be the string length of $S$. If $n=0$, then $S=\1_i$ for some $i\in \Z/m$. If $n=1$, then $S$ is string equivalent to an arrow or to the formal inverse of an arrow, which implies that there exists $i\in \Z/m$ such that $S\sim a_i$ or $S\sim \bar{a}_i$. Next assume that $n\geq 2$ and let $l$ be a maximal nonnegative integer such that for some $i\in \Z/m$, $(\bar{a}_ia_i)^l$ is string equivalent to a substring of $S$. If $l>0$, then there exist suitable strings $T$ and $T'$ for $\A_{m,N}$ such that $S\sim T(\bar{a}_ia_i)^l\bar{a}_iT'$, or $S\sim Ta_i(\bar{a}_ia_i)^lT'$, or $S\sim T(\bar{a}_ia_i)^lT'$. If $S\sim T(\bar{a}_ia_i)^l\bar{a}_iT'$ (resp. $S\sim Ta_i(\bar{a}_ia_i)^lT'$, resp. $S\sim T(\bar{a}_ia_i)^lT'$) then the maximality of $l$ implies that $(\bar{a}_ia_i)^l\bar{a}_i$ (resp. $a_i(\bar{a}_ia_i)^l$, resp. $(\bar{a}_ia_i)^l$) starts in a peak and ends in a deep as a substring of $S$ (as explained in \S \ref{ape2}), which implies that there exists a non-trivial canonical endomorphism of $M[S]$ as in (\ref{canhom}) that factors through the string $\A_{m,N}$-module $M[\bar{a}_i]$ (resp. $M[a_i]$, resp. $M[\1_i]$), and thus $\dim_\k\End_{\A_{m,N}}(M[S])\geq 2$, contradicting the hypothesis. It follows then that $l=0$, which implies that neither $\bar{a}_ia_i$ nor its formal inverse is a substring of $S$ for all $i\in \Z/m$. Similarly we obtain that neither $a_i\bar{a}_i$ nor its formal inverse is a substring of $S$ for all $i\in \Z/m$. Therefore, $S=w_nw_{n-1}\cdots w_1$, where for each $1\leq j\leq n$, $w_j$ or $w_j^{-1}$ is in $\{a_k, \bar{a}_k\}$ for a suitable $k\in \Z/m$. Assume first that $m$ is odd. If $n\geq m$, then there exists at least a canonical endomorphism of $M[S]$  as in (\ref{canhom}) that factors through a simple $\A_{m,N}$-module, which again contradicts the hypothesis, and thus $n< m$. This implies that $S$ is equivalent to a substring of $Z_i$ for some $i\in \Z/m$. Next assume that $m$ is even. If $n\leq m$, then as before, $S$ is equivalent to a substring of $Z_i$ for some $i\in \Z/m$. Assume next that $n>m$ and let $l\geq 0$ be maximal such that $S\sim TZ_i^{(l)}$ for a suitable string $T$, where $Z_i^{(l)}$ is as in (\ref{induczigzag}). Assume further that $T$ has positive string length. Since for all $k\in \Z/m$, $S$ does contain a substring equivalent neither to $a_k\bar{a}_k$ nor to $\bar{a}_ka_k$, it follows from the maximality of $l$ that $T$ is equivalent to $a_k$ or to $\bar{a}_k$ for some $k\in \Z/m$.  This implies that there exists a non-trivial canonical endomorphism of the string $\A_{m,N}$-module $M[S]$ as in (\ref{canhom}) that factors through a non-simple string $\A_{m,N}$-module $M[S']$ with  $S'$ a string representative for $\A_{m,N}$, and which is equivalent to a proper substring of $S$. This contradicts that $\End_{\A_{m,N}}(M[S])\cong \k$, and therefore $T$ has string length equal to zero.  Hence in this situation, it follows that $S$ is equivalent to $Z_i^{(l)}=W_i^{(l)}$ for some $l\geq 0$.  This finishes the proof of Proposition \ref{prop3.1}. 
\end{proof}

\begin{remark}\label{rem5}
Let $i\in \Z/m$ be fixed but arbitrary and let $\mathfrak{C}$ be a connected component of $\Gamma_s(\A_{m,N})$ containing a string $\A_{m,N}$-module $M[S]$ with $\End_{\A_{m,N}}(M[S])\cong \k$. 
\begin{enumerate}
\item Note that if $N=1$, then for all $l\geq 0$, the string $\A_{m,1}$-module $M[W_i^{(l)}]$ as in (\ref{stringW}) can be obtained either from $M[a_i]$ or from $M[\bar{a}_i]$ by taking hooks or co-hooks as in \S\ref{ape2}, and thus they belong to the component of the Auslander-Reiten quiver of $\A_{m,1}$ containing either $M[a_i]$ or $M[\bar{a}_i]$. In this situation, we obtain that $\mathfrak{C}$ contains either a simple $\A_{m,1}$-module or contains one of the string $\A_{m,1}$-modules $M[a_i]$ or $M[\bar{a}_i]$ with $i\in \Z/m$, and thus there are at most finitely many possibilities for $\mathfrak{C}$.
\item Assume now that $N\geq 2$. Then for all $l\geq 0$, the string $\A_{m,N}$-module $M[W_i^{(l)}]$ cannot be obtained from a string $\A_{m,N}$-module of minimal string length by adding hooks or co-hooks, which implies that  $M[W_i^{(l)}]$ is of minimal string length. 
\begin{enumerate}
\item If $m$ is odd, then $\mathfrak{C}$ contains either a simple $\A_{m,N}$-module, or one of the string $\A_{m,N}$-modules $M[a_i]$, or $M[\bar{a}_i]$ with $i\in \Z/m$, or the string $\A_{m,N}$-module $M[T]$, where $T$ is a substring of $Z_i$ for $i\in \Z/m$.  In this situation we also have that there are at most finitely many possibilities for $\mathfrak{C}$. 
\item Finally, by Proposition \ref{prop3.1}, if $m$ is even then for all $l\geq 0$, $M[W_i^{(l)}]$ has endomorphism ring isomorphic to $\k$, which implies that in this situation, there are infinitely many possibilities for $\mathfrak{C}$. 
\end{enumerate}
\end{enumerate} 
\end{remark}

Following Remark \ref{rem5}, we obtain the following result. 

\begin{corollary}\label{cor3.4}
\begin{enumerate}
\item If $N=1$, then there are at most finitely many components of $\Gamma_s(\A_{m,1})$ containing a string $\A_{m,1}$-module of minimal string length whose endomorphism ring is isomorphic to $\k$.
\item Assume that $N\geq 2$.
\begin{enumerate}
\item If $m$ is odd, then there are at most finitely many components  of type $\mathbb{ZA}_\infty^\infty$ of $\Gamma_s(\A_{m,N})$ containing a string $\A_{m,N}$-module $V$ with $\End_{\A_{m,N}}(V)\cong \k$.
 \item If $m$ is even, then there are infinitely many components of type $\mathbb{ZA}_\infty^\infty$ of $\Gamma_s(\A_{m,N})$ containing a string $\A_{m,N}$-module $V$ with $\End_{\A_{m,N}}(V)\cong \k$.
\end{enumerate}
\end{enumerate}
\end{corollary}



\section{Universal deformation rings for string $\A_{m,N}$-modules}\label{section4}
In this section with consider the components $\mathfrak{C}$ of the stable Auslander-Reiten quiver of $\A_{m,N}$ that contain a module whose endomorphism ring is isomorphic to $\k$ as in Proposition \ref{prop3.1}. 

Let $M$ and $N$ be arbitrary finitely generated $\A_{m,N}$-modules. Since $\A_{m,N}$ is self-injective, it follows that for all $j\geq 1$, $\Ext_{\A_{m,N}}^j(M,N)=\SHom_{\A_{m,N}}(\Omega^jM,N)$, where $\SHom_{\A_{m,N}}(\Omega^jM,N)$ is the quotient of $\k$-vector spaces of $\Hom_{\A_{m,N}}(\Omega^jM,N)$ by the subspace consisting in those morphisms from $\Omega^j M$ to $N$ that factor through a finitely generated projective $\A_{m,N}$-module.  

For all $i\in \Z/m$, we define
\begin{align*}
\underline{c}_i=a_{i-1}\bar{a}_{i-2}^{-1}(\bar{a}_{i-2}a_{i-2})^{1-N}, && \underline{d}_i=(a_{i+1}\bar{a}_{i+1})^{N-1}a_{i+1}\bar{a}_i^{-1}, 
\end{align*}
and for all $n\geq 1$, define the {\bf string} representative
\begin{align}\label{hooks}
\underline{C}_{i,n}&=\underline{w}_1\underline{w}_2\cdots\underline{w}_n,
\end{align}
where $\mathbf{t}(\underline{C}_i)=i$ and for all $1\leq j \leq n$, $\underline{w}_j=\underline{c}_k$ for a suitable $k\in\Z/m$. We also define the {\bf string} representative
\begin{align}\label{cohooks}
\underline{D}_{i,n}&=\underline{v}_1\underline{v}_2\cdots\underline{v}_n,
\end{align}
where $\mathbf{s}(\underline{D}_i)=i$ and for all  $1\leq j\leq n$, $\underline{v}_j=\underline{d}_k$ for a suitable $k\in \Z/m$. For all $i\in \Z/m$, we let $\underline{C}_{i,0}=\1_i=\underline{D}_{i,0}$. Note e.g. that if $i\in \Z/m$ is fixed and $n=2$, then $\underline{C}_{i,2}= (\1_i)_{hh}$, and $\underline{D}_{i,2}={_{hh}(\1_i)}$ (see \S\ref{ape2}). 

\subsection{Components of $\Gamma_s(\A_{m,N})$ of type $\mathbb{ZA}_\infty^\infty$ containing a module whose endomorphism ring is isomorphic to $\k$}\label{sect4}

\subsubsection{Components containing a simple $\A_{m,N}$-module}

\begin{proposition}\label{prop4}
Let $i\in \Z/m$ be fixed and let $\mathfrak{A}_i$ be the component of $\Gamma_s(\A_{m,N})$  containing the simple $\A_{m,N}$-module $M[\1_i]$. Then the component $\mathfrak{A}_i$ is not $\Omega$-stable, and all the $\A_{m,N}$-modules $V$ in $\mathfrak{A}_i\cup \Omega \mathfrak{A}_i$ have stable endomorphism ring $\SEnd_{\A_{m,N}}(V)$ isomorphic to $\k$.  Their universal deformation rings $R(\A_{m,N},V)$ are also isomorphic to $\k$.
\end{proposition}   
\begin{proof}
Let $V_0=M[\1_i]$. Then $\Omega V_0 = M[{_c}(\bar{a}_{i-1}a_{i-1})^{N-1}]$, where ${_c}(\bar{a}_{i-1}a_{i-1})^{N-1}$ is the left co-hook of $(\bar{a}_{i-1}a_{i-1})^{N-1}$ as in \S\ref{ape2}. This implies that $\Omega \mathfrak{A}_i\not=\mathfrak{A}_i$. On the other hand, for all $n\geq 1$, let $V_n=M[\underline{D}_{i,n}]$ and $V_{-n}=M[\underline{C}_{i,n}]$. It follows from Remark \ref{rem0} and \S\ref{ape2} that every string $\A_{m,N}$-module belonging to  $\mathfrak{A}_i\cup\Omega\mathfrak{A}_i$ lies in the $\Omega$-orbit of $V_n$ for some $n\in \Z$. 
Using the description of the canonical morphisms in \S\ref{appendix3} and the shape of the indecomposable $\A_{m,N}$-modules as in Figure \ref{projec} it follows that for all $n\in\Z$, $\SEnd_{\A_{m,N}}(V_n)\cong \k$, which implies that $V_n$ has a universal deformation ring $R(\A_{m,N}, V_n)$. Since it is straightforward to show by using \S\ref{appendix3} that $\Ext_{\A_{m,N}}^1(V_n,V_n)=\SHom_{\A_{m,N}}(\Omega V_n, V_n)=0$, we obtain that $R(\A_{m,N}, V_n)=\k$ for all $n\in \Z$. This finishes the proof of Proposition \ref{prop4}.
\end{proof}
\begin{remark}
If $m=3$ and $N=1$, then Proposition \ref{prop4} recovers some of the results in \cite[Prop. 3.1(iii)]{bleher9} for the algebra $\A_{3,1}=D(3\mathcal{K})^{1,1,1}$. 
\end{remark}

\subsubsection{Components containing a string $\A_{m,N}$-module of the form $M[a_i]$ or $M[\bar{a}_i]$}\label{ai}
\renewcommand{\qedsymbol}{}

Let $\kappa_m$ be the non-negative integer defined as follows: 

\begin{equation}\label{kappam}
\kappa_m=\begin{cases} \frac{m}{2}, &\text{ if $m$ is even,}\\
\frac{m-1}{2}, &\text{ otherwise.}
\end{cases}
\end{equation} 

\begin{remark}\label{rem42}
If $N=1$, then for all $i\in \Z/m$, the strings $a_i$ and $\bar{a}_i$ for $\A_{m,1}$ are maximal. Thus the respective components of $\Gamma_s(\A_{m,1})$ containing the string $\A_{m,1}$-modules $M[a_i]$ and $M[\bar{a}_i]$ with $i\in \Z/m$ are tubes, which will be discussed in Proposition \ref{prop6}. Thus, for the remainder of \S\ref{sect4}, we assume that $N\geq 2$.
\end{remark}

\begin{proposition}\label{prop5}
Let $i\in \Z/m$ be fixed. Let $\mathfrak{B}_i$ be the component of $\Gamma_s(\A_{m,N})$  containing the $\A_{m,N}$-module $M[\zeta_i]$ with $\zeta_i\in \{a_i,\bar{a}_i\}$. 
Then the component $\mathfrak{B}_i$ is not $\Omega$-stable unless $N=2$, and there are exactly $\kappa_m$ $\Omega$-orbits of $\A_{m,N}$-modules in $\mathfrak{B}_i\cup \Omega\mathfrak{B}_i$ whose stable endomorphism ring is isomorphic to $\k$. More precisely,  the modules in $\mathfrak{B}_i\cup \Omega\mathfrak{B}_i$ that have stable endomorphism ring is isomorphic to $\k$ are the modules in the $\Omega$-orbits of $V_n=M[\underline{D}_{i+1,n}\zeta_i]$ with $0\leq n\leq \kappa_m-1$. The universal deformation rings $R(\A_{m,N},V_n)$ are isomorphic to $\k[[t]]/(t^N)$ when $n=0$, and to $\k$ when $0<n<\kappa_m-1$. If $m$ is odd, then $R(\A_{m,N}, V_{\kappa_m-1})$ is isomorphic to $\k$. If $m$ is even, then $R(\A_{m,N}, V_{\kappa_m-1})$ is isomorphic to $\k[[t]]$.
\end{proposition}

\begin{proof}
Let $V_0=M[a_i]$. By using the shape of the indecomposable $\A_{m,N}$-modules as in Figure \ref{projec}, we obtain that $\Omega V_0=\Omega M[a_i]= M[{_c}(a_i(\bar{a}_ia_i)^{N-2})]$, where ${_c}(a_i(\bar{a}_ia_i)^{N-2})$ is the left co-hook of the string $a_i(\bar{a}_ia_i)^{N-2}$ as in \S\ref{ape2}. This implies that $\mathfrak{B}_i$ is not $\Omega$-stable unless $N=2$. For all $n\geq 1$, let $V_n=M[\underline{D}_{i+1,n}a_i]$ and $V_{-n}=M[a_i\underline{C}_{i,n}]$.  Note that for example that if $n=2$, then $V_2=M[{_{hh}(a_i)}]$ and $V_{-2}=M[(a_i)_{hh}]$. 
As before, it follows from Remark \ref{rem1} and \S\ref{ape2} that every string $\A_{m,N}$-module belonging to  $\mathfrak{B}_i\cup\Omega\mathfrak{B}_u$ lies in the $\Omega$-orbit of $V_n$ for some $n\in \Z$. If $N=2$, then for all $n\geq 1,$ the $\A_2$-module $V_{-n}$ lies in the $\Omega$-orbit of $V_{k_n}$ for some integer $k_n\geq 0$. If $N\geq 3$ and $n\geq 1$, it follows from the description of the canonical morphisms in \S\ref{appendix3} and the shape of the indecomposable $\A_{m,N}$-modules as in Figure \ref{projec} that the $\A_{m,N}$-module $V_{-n}$ has a canonical endomorphism as in (\ref{canhom}) factoring through the simple $\A_{m,N}$-module $M[\1_{i+1}]$, and which does not factor through a projective $\A_{m,N}$-module. Therefore in this situation we have $\SEnd_{\A_{m,N}}(V_{-n})\not\cong \k$. On the other hand, it is straightforward to show that for all $N\geq 2$, $\SEnd_{\A_{m,N}}(V_n)\cong \k$ for $1\leq n\leq \kappa_m-1$. Therefore if $N\geq 2$ and $0\leq n\leq \kappa_m-1$, then $V_n$ has a universal deformation ring $R(\A_{m,N}, V_n)$. On the other hand, we have that $\SEnd_{\A_{m,N}}(V_n)\not\cong \k$ for $n\geq \kappa_m$, for in this situation $V_n$ has a non-trivial canonical endomorphism as in (\ref{canhom}) that factors through either a simple $\A_{m,N}$-module or a string $\A_{m,N}$-module whose string representative has string length $1$, and which does not factor through a projective $\A_{m,N}$-module. 
Assume that $n=0$. Then there is a unique canonical morphism in $\Hom_{\A_{m,N}}(\Omega V_0, V_0)$ that factors through $M[a_i]$ which does not factor through a projective $\A_{m,N}$-module. Since 
$\dim_\k\Hom_{\A_{m,N}}(\Omega V_0, V_0)=1,$
it follows that $\Ext_{\A_{m,N}}^1(V_0,V_0) \cong \SHom_{\A_{m,N}}(\Omega V_0, V_0)\cong \k$. Therefore, $R(\A_{m,N}, V_0)$ is a quotient of $\k[[t]]$. 

\begin{claim}\label{claim1}
$R(\A_{m,N}, V_0)=R(\A_{m,N}, M[a_i])$ is isomorphic to  $\k[[t]]/(t^N)$.
\end{claim}

\begin{proof}
For all $0\leq l\leq N-1$, let $S_l=(a_i\bar{a}_i)^la_i$. Let  $l\in\{1,\ldots,N\}$ be fixed. Then there exists a non-trivial canonical endomorphism $\sigma_l$ of the $\A_{m,N}$-module $M[S_l]$ as in (\ref{canhom}) that factors through $M[S_{l-1}]$, namely
\begin{equation}\label{sigma1}
\sigma_l: M[S_l]\surjection M[S_{l-1}]\injection M[S_l]. 
\end{equation}
Observe that  the kernel of $\sigma_l$ as well as the image of $\sigma_l^{l-1}$ are isomorphic to $V_0=M[a_i]$, and that $\sigma_l^l=0$. Thus, the $\A_{m,N}$-module $M[S_l]$ is naturally a $\k[[t]]/(t^{l+1})\otimes_\k\A_{m,N}$-module, where the action of $t$ on $x\in M[S_l]$ is given as $t\cdot x=\sigma_l(x)$. In particular, $tM[S_l]\cong M[S_{l-1}]$. 

Let  $\{\bar{r}_1, \bar{r}_2\}$ be a $\k$-basis of $V_0$. Using the 
isomorphism $M[S_l]/tM[S_l]\cong V_0$, we can lift the elements $\bar{r}_1$ and $\bar{r}_2$ to corresponding elements $r_1, r_2 \in M[S_l]$. It follows that $\{r_1, r_2\}$ is linearly independent over $\k$ and that $\{t^sr_1, t^sr_2: 1\leq 
s\leq l\}$ is a $\k$-basis of $tM[S_l]\cong M[S_{l-1}]$. Therefore, $\{r_1, r_2\}$ is a $\k[[t]]/(t^{l+1})$-basis of $M[S_l]$, i.e., $M[S_l]$ is free over $\k[[t]]/(t^{l+1})$. Observe that $M[S_l]$ lies in a short exact sequence of $\A_{m,N}$-modules 

\begin{equation*}
0\to tM[S_l]\to M[S_l]\to \k\otimes_{\k[[t]]/(t^{l+1})}M[S_l]\to 0.
\end{equation*}

Then there exists an isomorphism of $\A_{m,N}$-modules $\phi_l:\k\otimes_{\k[[t]]/(t^{l+1})}M[S_l]\to V_0$, which implies that $(M[S_l],\phi_l)$ is a lift of $V_0$ over $k[[t]]/(t^{l+1})$. Consider the lift $(M[S_{N-1}],\phi_{N-1})$ of $V_0$ over $\k[[t]]/(t^N)$. Since $\End_{\A_{m,N}}(V_0)\cong \k$, it follows from Theorem \ref{thm3}(i) that there exists a unique morphism $\alpha:R(\A_{m,N},V_0)\to \k[[t]]/(t^N)$ in 
$\hat{\Ca}$ such that 
\begin{equation*}
M[S_{N-1}]\cong \k[[t]]/(t^N)\otimes_{R(\A_{m,N},V_0),\alpha} U(\A_{m,N},V_0),
\end{equation*} 
where $R(\A_{m,N}, V_0)$ and $[U(\A_{m,N},V_0), \phi_{U(\A_{m,N},V_0)}]$ are respectively the universal deformation ring and the universal deformation of the $\A_{m,N}$-module $V_0$. Since $(M[S_1],\phi_1)$ is not the trivial lift of $V_0$ over $\k[[t]]/(t^2)$, it follows that there exists 
a unique surjective morphism $\alpha': R(\A_{m,N},V_0)\to \k[[t]]/(t^2)$ in $\hat{\Ca}$ such that
\begin{equation*}
M[S_1]\cong k[[t]]/(t^2)\otimes_{R(\A_{m,N},V_0),\alpha'}U(\A_{m,N},V_0). 
\end{equation*}

By considering the natural projection $\pi_{N,2}: \k[[t]]/(t^N)\to \k[[t]]/(t^2)$ and the lift $(U',\phi_{U'})$  of $V_0$ over $\k[[t]]/(t^2)$ corresponding to the morphism $\pi_{N,2}\circ \alpha$, we obtain 
\begin{align*}
U'&\cong \k[[t]]/(t^2)\otimes_{R(\A_{m,N},V_0),\pi_{N,2}\circ \alpha}U(\A_{m,N},V_0)\\
&\cong \k[[t]]/(t^2)\otimes_{\k[[t]]/(t^N),\pi_{N,2}}\left(\k[[t]]/(t^N)\otimes_{R(\A_{m,N},V_0),\alpha}U(\A_{m,N},V_0)\right)\\
&\cong \k[[t]]/(t^2)\otimes_{\k[[t]]/(t^N),\pi_{N,2}}M[S_{N-1}]\\
&\cong M[S_{N-1}]/t^2M[S_{N-1}]\cong M[S_1]. 
\end{align*} 
It follows from Remark \ref{rem1} that $[U',\phi_{U'}]=[M[S_1],\phi_1]$ in $\hat{\Fun}_{V_0}(\k[[t]]/(t^2))$. The uniqueness of $\alpha'$ implies that $\alpha'= \pi_{N,2}\circ \alpha$. 

Since $\alpha'$ is surjective, $\alpha$ is also surjective.  We want to prove that $\alpha$ is an isomorphism. Suppose this is false.  Then there exists a surjective $\k$-algebra homomorphism $\alpha_0: R(\A_{m,N},V_0)\to \k[[t]]/(t^{N+1})$ in $\hat{\Ca}$ such that $\pi_{N+1,N}\circ \alpha_0=\alpha$, where $\pi_{N+1,N}:\k[[t]]/(t^{N+1})\to \k[[t]]/(t^{N})$ is the natural projection. Let $M_0$ be a $\k[[t]]/(t^{N+1})\otimes_\k\A_{m,N}$-module, which defines the lift of $V_0$ over $\k[[t]]/(t^{N+1})$ corresponding to $\alpha_0$. Since the kernel of 
$\pi_{N+1,N}$ is $(t^N)/(t^{N+1})$, then $M_0/t^{N}M_0\cong M[S_{N-1}]$. Consider the $\k[[t]]/(t^{N+1})\otimes_\k\A_{m,N}$-module homomorphism $g:M_0\to t^NM_0$ defined by $g(x)=t^Nx$ for all $x\in M_0$. Since $M_0$ is free over $\k[[t]]/(t^{N+1})$, then the kernel of $g$ is isomorphic to $tM_0$. Thus, $M_0/tM_0\cong t^NM_0$ for $g$ is a surjection. Therefore $t^NM_0\cong V_0$, which implies that there exists a non-split short exact sequence of $\k[[t]]/(t^{N+1})\otimes_\k \A_{m,N}$-modules 
\begin{equation}\label{seq1}
0\to V_0\to M_0\to M[S_{N-1}]\to 0.
\end{equation}

Since $\Omega M[S_{N-1}] = \Omega M[(a_i\bar{a}_i)^{N-1}a_i]=M[(\bar{a}_{i-1}a_{i-1})^{N-1}\bar{a}_{i-1}]$, then 

\begin{equation*}
\Ext_{\A_{m,N}}^1(M[S_{N-1}],V_0)=\underline{\Hom}_{\A_{m,N}}(\Omega M[S_{N-1}],V_0)=0.
\end{equation*}

It follows that the sequence (\ref{seq1}) splits as a sequence of $\A_{m,N}$-modules. Hence, $M_0=V_0\oplus M[S_{N-1}]$ as $\A_{m,N}$-modules. Identifying the elements of $M_0$ as $(v,x)$ with $v\in V_0$ and $x\in M[S_{N-1}]$, we see that the $t$ acts on $(v,x)\in M_0$ as $t\cdot(v,x)=(\mu(x),\sigma_{N-1}(x))$, where $\mu: M[S_{N-1}]\to V_0$ is a surjective $\A_{m,N}$-module homomorphism and $\sigma_{N-1}$ is as in (\ref{sigma1}). Since the canonical homomorphism $\epsilon: M[S_{N-1}] \surjection M[a_i]$ generates $\Hom_{\A_{m,N}}(M[S_{N-1}],V_0)$, then there exists $c\in \k^\ast$ such that $\mu=c\epsilon$, which implies that the kernel of $\mu$ is $tM[S_{N-1}]$. Therefore, $t^N(v,x)=(\mu(t^{N-1}x),\sigma_{N-1}^N(x))=(0,0)$ for all $v\in V_0$ and $x\in M[S_{N-1}]$, which contradicts the fact that $t^NM_0\cong V_0$. Thus $\alpha: R(\A_{m,N},V_0) \to \k[[t]]/(t^N)$ is an isomorphism 
and $R(\A_{m,N},M[a_i])=R(\A_{m,N},V_0)\cong \k[[t]]/(t^N)$.  This finishes the proof of Claim \ref{claim1}.
\end{proof}
Assume now that $1\leq n<\kappa_m-1$. Then it is straightforward to show that  $\Ext_{\A_{m,N}}^1(V_n,V_n)=0$, which implies that $R(\A_{m,N}, V_n)\cong \k$.
Finally assume that $n=\kappa_m-1$. If $m$ is odd, then $\Ext_{\A_{m,N}}^1(V_{\kappa_m-1},V_{\kappa_m-1})=0$, which implies that $R(\A_{m,N}, V_{\kappa_m-1})\cong \k$. Assume next that $m$ is even. Then $\Hom_{\A_{m,N}}(\Omega V_{\kappa_m-1}, V_{\kappa_m-1})$ has only a non-trivial canonical morphism as in (\ref{canhom}) which does not factor through a projective $\A_{m,N}$-module. Therefore 
\begin{equation*}
\Ext_{\A_{m,N}}^1(V_{\kappa_m-1},V_{\kappa_m-1})=\SHom_{\A_{m,N}}(\Omega V_{\kappa_m-1}, V_{\kappa_m-1})=\k.
\end{equation*}
This implies that $R(\A_{m,N}, V_{\kappa_m-1})$ is isomorphic to a quotient of $\k[[t]]$.

\begin{claim}\label{claim2}
If $m$ is even then $R(\A_{m,N},V_{\kappa_m-1})$ is isomorphic to $\k[[t]]$.
\end{claim}
\begin{proof}
Let $W_i= \underline{D}_{i+1,\kappa_m-1}a_i$. Note that $\mathbf{s}(W_i)=i$, $\mathbf{t}(W_i)=i-1$ and $V_{\kappa_m-1}=M[W_i]$. Let $T_0=W_i$ and for all $l\geq 1$, let $T_l$ be the string 
\begin{align*}
T_l =T_{l-1}\bar{a}_{i-1}^{-1}W_i
\end{align*}
By using similar arguments as  
in the proof of Claim \ref{claim1}, for each $l\geq 1$ we get lifts $(M[T_l], \varphi_l)$ of $M[W_i]$ over $\k[[t]]/(t^{l+1})$, where $t$ acts on $x\in M[T_l]$ as $t\cdot x =\delta_l(x)$ and where $\delta_l$ is the non-trivial canonical endomorphism of $M[T_l]$ that factors through $M[T_{l-1}]$, namely 

\begin{equation}\label{sigma2}
\delta_l: M[T_l]\surjection M[T_{l-1}]\injection M[T_l].
\end{equation}

Note that for all $l\geq 1$, we have natural projections $\pi_{l,l-1}: M[T_l]\to M[T_{l-1}]$. Let $N_0=\invlim M[T_l]$ and let $t$ act on $N_0$ as $\invlim \pi_{l,l-1}$. In particular, $N_0$ is a $\k[[t]]\otimes_\k\A_{m,N}$-module and $\k
\otimes_{\k[[t]]}N_0\cong N_0/tN_0\cong M[W_i]$, which implies that there exists an isomorphism of $\A_{m,N}$-modules $\varphi_0:\k\otimes_{\k[[t]]}N_0\to M[W_i]$. Let $k=\dim_\k M[W_i]$ and let $\{\bar{b}_j\}_{1\leq j\leq k}$ be a $\k$-basis of $N_0/tN_0$. For all $1\leq j\leq k$, we are able to lift these elements $\bar{b}_j$ in $N_0/tN_0$ to elements $b_j$ of 
$N_0$ such that $\{b_j\}_{1\leq j\leq k}$ is a generating set of the $\k[[t]]\otimes_\k\A_{m,N}$-module $N_0$. It follows that $\{b_j\}_{1\leq j\leq n}$ is a $\k[[t]]$-basis 
of $N_0$, i.e., $N_0$ is free over $\k[[t]]$. Therefore, $(N_0,\varphi_0)$ is a lift of $M[W_i]$ over $\k[[t]]$ and there exists a unique $\k$-algebra homomorphism $\beta:R(\A_{m,N},M[W_i])\to \k[[t]]$ in $\hat{\Ca}$ corresponding to the deformation defined by $(N_0,\varphi_0)$.
Since $N_0/t^2N_0\cong M[T_1]$ as $\A_{m,N}$-modules, we can see, as in the proof of Claim \ref{claim1}, that $N_0/t^2N_0$ defines a non-trivial lift of $M[W_i]$ over $\k[[t]]/(t^2)$ and that $\beta$ is a surjection. Since $R(\A_{m,N},M[W_i])$ is a quotient of $\k[[t]]$, it follows that $\beta$ is an isomorphism. Hence $R(\A_{m,N},V_{\kappa_m-1})=R(\A_{m,N},M[W_i])\cong \k[[t]]$. This finishes the proof of Claim \ref{claim2}.
\end{proof}

Note that the above results concerning the string $\A_{m,N}$-module $M[a_i]$ can be adjusted to obtain those for $M[\bar{a}_i]$. In particular, $R(\A_{m,N}, M[\bar{a}_i])\cong \k[[t]]/(t^N)$. This finishes the proof of Proposition \ref{prop5}. 
\renewcommand{\qedsymbol}{$\Box$}
\end{proof}   

\subsubsection{Components containing a string $\A_{m,N}$-module whose endomorphism ring is isomorphic to $\k$, and whose string representative have string length greater than one.}   

For all integers $l\geq 1$ and $i\in \Z/m$, define inductively the {\bf string} representative $\Theta_{i,l}$ for $\A_{m,N}$ as follows: 
\begin{align}\label{Thetastrings}
\Theta_{i,1}&\in \{\bar{a}_{i+1}^{-1}a_i, a_{i+1}\bar{a}_i^{-1}\},\\\notag
\Theta_{i,l+1}&= \Theta_{\mathbf{t}(\Theta_{i,l}),1}\Theta_{i,l}, \text{ for $l\geq 1$.}
\end{align}

\begin{proposition}\label{prop7}
Let $S$ be a string representative for $\A_{m,N}$ with string length greater than one, and which is equivalent either to a substring of $Z_i$ or to $W_i^{(n)}$ as in (\ref{stringW}) for some $i\in \Z/m$ and $n\geq 0$. Let $\mathfrak{C}$ be the component of $\Gamma_s(\A_{m,N})$ containing the string $\A_{m,N}$-module $M[S]$. 
\begin{enumerate}
\item If $S$ is string equivalent to $\Theta_{j,l}$ (as in  (\ref{Thetastrings})) for some $j\in \Z/m$ and $l\geq 1$, then every $\A_{m,N}$-module $V$ in $\mathfrak{C}\cup\Omega\mathfrak{C}$ has stable endomorphism ring isomorphic to $\k$.  In this situation, the universal deformation ring $R(\A_{m,N},V)$ is isomorphic to $\k$.
\item If $S$ is not string equivalent to $\Theta_{j,l}$ for all $j\in \Z/m$ and $l\geq 1$, then there are exactly $2\kappa_m$ $\Omega^2$-orbits of $\A_{m,N}$-modules $V$ in $\mathfrak{C}\cup\Omega\mathfrak{C}$ that have stable endomorphism ring isomorphic to $\k$.  In this situation, the universal deformation ring $R(\A_{m,N},V)$ is isomorphic either to $\k$ or to $\k[[t]]$.
\end{enumerate}
\end{proposition}
\begin{proof}
First assume that $S$ is equivalent to $\Theta_{i,l}$ for some $i\in \Z/m$ and $l\geq 1$. Note that every $\A_{m,N}$-module in $\mathfrak{C}\cup\Omega\mathfrak{C} $ lies in the $\Omega$-orbit of either $M[S]$, or $M[S\underline{C}_{i,n}]$, or $M[\underline{D}_{\mathbf{t}(S),n}S]$ for some $n\geq 1$.  By using the shape of the indecomposable projective $\A_{m,N}$-modules as in Figure \ref{projec} together with the description of the canonical endomorphisms in \S\ref{appendix3}, it is straightforward to show that $\SEnd_{\A_{m,N}}(M[S\underline{C}_{i,n}])\cong \k\cong \SEnd_{\A_{m,N}}(M[\underline{D}_{\mathbf{t}(S),n}S])$. Since $\A_{m,N}$ is self-injective, it follows that $\Omega$ induces a self-equivalence of the stable module category of $\A_{m,N}$. Then $\SEnd_{\A_{m,N}}(\Omega M[S\underline{C}_{i,n}])\cong \k\cong \SEnd_{\A_{m,N}}(\Omega M[\underline{D}_{\mathbf{t}(S),n}S])$, which shows that every $\A_{m,N}$-module in $\mathfrak{C}\cup\Omega\mathfrak{C}$ has stable endomorphism ring isomorphic to $\k$. Assume next that 
\begin{equation}\label{descS}
S\sim \cdots \bar{a}_{i+1}^{-1}a_i.
\end{equation}
It follows that $\Omega M[S]\cong M[\Delta S]$, where 
\begin{equation}
\Delta S\sim \cdots (a_{i+1}\bar{a}_{i+1})^{N-1}a_{i+1}(a_i\bar{a}_i)^{1-N}\bar{a}_i^{-1}(a_{i-1}\bar{a}_{i-1})^{N-1}a_{i-1}.
\end{equation}
Again by using the shapes of the indecomposable projective $\A_{m,N}$-modules as in Figure \ref{projec} together with the description of the canonical homomorphisms as in \S\ref{appendix3}, it follows that 
\begin{align*}
\Ext_{\A_{m,N}}^1(M[S], M[S])&=\SHom_{\A_{m,N}}(\Omega M[S], M[S])\\
&=\SHom_{\A_{m,N}}(M[\Delta S], M[S]) =0, 
\end{align*}
which implies that $R(\A_{m,N}, M[S])=\k$.
Similarly, since for all $n\geq 1$, 
\begin{align*}
\Ext_{\A_{m,N}}^1(M[S\underline{C}_{i,n}], M[S\underline{C}_{i,n}])=0=\Ext_{\A_{m,N}}^1(M[\underline{D}_{\mathbf{t}(S),n}S],M[\underline{D}_{\mathbf{t}(S),n}S]), 
\end{align*} 
we obtain that $R(\A_{m,N},M[S\underline{C}_{i,n}])=\k= R(\A_{m,N}, M[\underline{D}_{\mathbf{t}(S),n}S])$. Note that we can adjust the above arguments for the case 
\begin{equation}\label{descS}
S\sim \cdots a_{i+1}\bar{a}_i^{-1}.
\end{equation}
Since the isomorphism class of universal deformation rings is invariant under $\Omega$ by Theorem \ref{thm3}(iii), we obtain Proposition \ref{prop7}(i).

To prove Proposition \ref{prop7}(ii), assume that $\mathbf{s}(S)=i$ and that $S$ is not equivalent to $\Theta_{i,l}$ for all $l\geq 1$. Let $n_1\geq 0$ maximal such that the string $\Theta_{i, n_1}$ is string equivalent to a substring of $S$. Then there exists a string $\zeta$ of length $1$ with $\mathbf{s}(\zeta)= \mathbf{t}(\Theta_{i, n_1})$ such that $S\sim \zeta \Theta_{i, n_1}$. Let $\kappa_m$ be as in (\ref{kappam}). By using the shape of the indecomposable projective $\A_{m,N}$-modules as in 
Figure \ref{projec} together with the description of the canonical endomorphisms in \S\ref{appendix3}, it is straightforward to show that for all $0\leq j< \kappa_m$, $\SEnd_{\A_{m,N}}(M[\underline{D}_{\mathbf{t}(S),j}S])\cong \k\cong \SEnd_{\A_{m,N}}(M[S\underline{C}_{i,j}])$, and that if $j\geq \kappa_m$,  $M[\underline{D}_{\mathbf{t}(S),j}S]$ (resp. $M[S\underline{C}_{\mathbf{s}(S),j}]$) has a canonical endomorphism ring, which does not factor through a projective $\A_{m,N}$-module.  Thus there are exactly $2\kappa_m$ $\Omega^2$-orbits in $\mathfrak{C}\cup \Omega\mathfrak{C}$ of $\A_{m,N}$-modules whose stable endomorphism ring is isomorphic to $\k$.

Let $0\leq j <\kappa_m-1$  be fixed but arbitrary, and  let $S'=\underline{D}_{\mathbf{t}(S),j}S$ (resp. $S''=S\underline{C}_{i,j}$). As before, it is straightforward to show by using the shape of the indecomposable projective $\A_{m,N}$-modules as in Figure \ref{projec} together with the description of the canonical endomorphisms in \S\ref{appendix3} that  $\Ext_{\A_{m,N}}^1(M[S'],M[S'])=0=\Ext_{\A_{m,N}}^1(M[S''],M[S''])$. This implies $R(\A_{m,N},M[S'])\cong \k \cong R(\A_{m,N},M[S''])$. 

Next let $S'=\underline{D}_{\mathbf{t}(S),\kappa_m-1}S$ (resp. $S''=S\underline{C}_{i,\kappa_m-1}$). Then 
\begin{equation*}
\Ext_{\A_{m,N}}^1(M[S'],M[S'])\cong \k\cong \Ext_{\A_{m,N}}^1(M[S''],M[S'']),
\end{equation*}
which implies that $R(\A_{m,N},M[S'])$ and  $R(\A_2,M[S''])$ are quotients of $\k[[t]]$. Let $T'_0= S'$ (resp. $T''_0=S''$) and for all  $l\geq 1$, let $T'_l=T'_{l-1}\gamma' S'$ (resp. $T''_l=T''_{l-1}\gamma'' S'$) where $\gamma'$ (resp. $\gamma''$) is either an arrow or the formal inverse of an arrow joining $\mathbf{t}(S')$ with $\mathbf{s}(S')-1$ (resp. $\mathbf{t}(S'')$ with $\mathbf{s}(S'')-1$) $\mod m$.  By using similar arguments as  in the proof of Claim \ref{claim1}, for each $l\geq 1$, we get lifts $(M[T'_l], \varphi'_l)$ (resp.$(M[T''_l], \varphi''_l)$) of $M[S']$ (resp. $M[S'']$) over $\k[[t]]/(t^{l+1})$, where $t$ acts on $x\in M[T'_l]$ (resp. $x\in M[T''_l]$) as $t\cdot x =\delta'_l(x)$ (resp. $t\cdot x =\delta''_l(x)$), and where $\delta'_l$ (resp. $\delta''$) is the non-trivial canonical endomorphism of $M[T'_l]$ (resp. $M[T''_l]$) that factors through $M[T_{l-1}]$ (resp. $M[T'_l]$).
By using similar arguments to those in the proof of Claim \ref{claim2}, we obtain that $R(\A_2, M[S'])\cong \k[[t]]\cong R(\A_2, M[S''])$.
This finishes the proof of Proposition \ref{prop7}.
\end{proof}

\subsection{Tubes of $\Gamma_s(\A_{m,N})$ containing string $\A_{m,N}$-modules.}
Using the description of the indecomposable projective $\A_{m,N}$-module as in Figure \ref{projec} and the definition of the string $\A_{m,N}$-modules in \S\ref{ape1}, it follows that if $m$ is odd, then there are exactly two $m$-tubes in $\Gamma_s(\A_{m,N})$, namely, one containing $M[(a_0\bar{a}_0)^{N-1}a_0]$  and another one containing $\Omega M[(a_0\bar{a}_0)^{N-1}a_0] = M[(\bar{a}_{m-1}a_{m-1})^{N-1}\bar{a}_{m-1}]$. 

On the other hand if $m$ is even, then there are exactly four $\frac{m}{2}$-tubes in $\Gamma_s(\A_{m,N})$, namely, one containing the string $\A_{m,N}$-module $M[(a_0\bar{a}_0)^{N-1}a_0]$, other containing $\Omega M[(a_0\bar{a}_0)^{N-1}a_0] = M[(\bar{a}_1a_1)^{N-1}\bar{a}_1]$, other containing the string $\A_{m,N}$-module $M[(\bar{a}_{m-1}a_{m-1})^{N-1}\bar{a}_{m-1}]$, and another one containing $\Omega M[(\bar{a}_{m-1}a_{m-1})^{N-1}\bar{a}_{m-1}]= M[(\bar{a}_0a_0)^{N-1}\bar{a}_0]$.  

As noted already in Remark \ref{rem42}, if $N=1$, then for all $i\in \Z/m$, the component of $\Gamma_s(\A_{m,1})$ containing the string $\A_{m,1}$-module $M[a_i]$ (resp. $M[\bar{a}_i]$) is a tube. 

\begin{proposition}\label{prop6}
Let $\mathfrak{T}$ be a tube of $\Gamma_s(\A_{m,N})$ containing only string $\A_{m,N}$-modules, and let $\kappa_m$ be as in (\ref{kappam}). Then there are exactly $\kappa_m$ $\Omega^2$-orbits of $\A_{m,N}$-modules in $\mathfrak{T}\cup\Omega\mathfrak{T}$ whose stable endomorphism ring is isomorphic to $\k$. Their universal deformation rings are isomorphic either to $\k$ or to $\k[[t]]$.
\end{proposition}   
\begin{proof}
By using Theorem \ref{thm3} together with Remark \ref{rem0}, we can prove Proposition \ref{prop6} by just considering the tube $\mathfrak{T}$ containing the string $\A_{m,N}$-module $V_0=M[(a_0\bar{a}_0)^{N-1}a_0]$ which lies on the boundary of $\mathfrak{T}$, for $(a_0\bar{a}_0)^{N-1}a_0$ is a maximal string for $\A_{m,N}$. For all $n\geq 1$, let $V_n=M[\underline{D}_{1,n}(a_0\bar{a}_0)^{N-1}a_0]$, where $\underline{D}_{1,n}$ is as in (\ref{cohooks}). Note that every $\A_{m,N}$-module in $\mathfrak{T}\cup\Omega\mathfrak{T} $ lies in the $\Omega$-orbit of $V_n$ for some $n\geq 0$. If $0\leq j\leq \kappa_m-1$, then by using the description of the indecomposable $\A_{m,N}$-modules together with that of the canonical morphisms in \S\ref{appendix3}, it is straightforward to show that $\SEnd_{\A_{m,N}}(V_j)\cong \k$. If $j\geq \kappa_m$, then the string $\A_{m,N}$-module $V_j$ has an a canonical endomorphism as in (\ref{canhom}) that factors through either a simple $\A_{m,N}$-module or a string $\A_{m,N}$-module corresponding to a maximal string for $\A_{m,N}$, and which does not factor through a projective $\A_{m,N}$-module. In this situation we have that $\SEnd_{\A_{m,N}}(V_j)\not\cong \k$, and therefore there are exactly $\kappa_m$ $\Omega^2$-orbits of $\A_{m,N}$-modules in $\mathfrak{T}\cup \Omega \mathfrak{T}$ whose stable endomorphism ring is isomorphic to $\k$. If $0\leq j< \kappa_m-1$ we have $\Ext_{\A_{m,N}}^1(V_j,V_j) = \SHom_{\A_{m,N}}(\Omega V_j,V_j)=0$, which implies $R(\A_{m,N}, V_j)\cong \k$. If $m$ is odd, then we obtain that 
\begin{equation*}
\Ext_{\A_{m,N}}^1(V_{\kappa_m-1},V_{\kappa_m-1}) = \SHom_{\A_{m,N}}(\Omega V_{\kappa_m-1},V_{\kappa_m-1})=0, 
\end{equation*}  
which implies $R(\A_{m,N}, V_{\kappa_m-1})\cong \k$. Assume that $m$ is even. Then there is a canonical morphism in $\Hom_{\A_{m,N}}(\Omega V_{\kappa_m-1},V_{\kappa_m-1})$ as in (\ref{canhom}) that factors through a simple $\A_{m,N}$-module, and which  does not factor through a projective $\A_{m,N}$-module. This shows that 
\begin{equation*}
\Ext_{\A_{m,N}}^1( V_{\kappa_m-1},V_{\kappa_m-1})\cong \k, 
\end{equation*}
which implies that $R(\A_{m,N}, V_{\kappa_m-1})$ is a quotient of $\k[[t]]$. By using analogous arguments to those in the proof of Claim \ref{claim2}, we obtain that $R(\A_{m,N}, V_{\kappa_m-1})\cong \k[[t]]$. This finishes the proof of Proposition \ref{prop6}.
\end{proof}

\begin{remark}
If $m=3$ and $N=1$, then Proposition \ref{prop6} recovers the results in \cite[Prop. 3.2(iii)]{bleher9} for the algebra $\A_{3,1}=D(3\mathcal{K})^{1,1,1}$.
\end{remark}


\subsection*{Acknowledgements}  
This article was developed during the visit of the fourth author to the Instituto de Matem\'aticas at the Universidad  de Antioquia in Medell{\'\i}n, Colombia during Summer 2016. The fourth author would like to express his gratitude to the other authors, faculty members, staff and students at the Instituto of Matem\'aticas as well as the other people related to this work at the Universidad de Antioquia for their hospitality and support during his visit. All the authors would like to express their gratitude to Professor Rachael Taillefer for providing some of the information stated in Remark \ref{rem0}.

\appendix

\section{Some Remarks about the Representation Theory of $\A_{m,N}$}\label{appendix}

Let $\A_{m,N}$ be as in Figure \ref{fig1}, where $m\geq 3$ and $N\geq 1$. We identify the vertices of $Q$ with elements of $\mathbb{Z}/m$ (the cyclic group of $m$ elements). Since $\A_{m,N}$ is a special biserial algebra (as introduced in \cite{wald}), all the non-projective indecomposable $\A_{m,N}$-modules can be described combinatorially by using so-called string and bands for $\A_{m,N}$. The corresponding indecomposable $\A_{m,N}$-modules are called string and band modules. In the following, we describe these string $\A_{m,N}$-modules, the components of the stable Auslander-Reiten quiver $\Gamma_s(\A_{m,N})$ containing them as determined in \cite{buri}, and the morphisms between them as determined in \cite{krause}.

\subsection{String modules for $\A_{m,N}$}\label{ape1}

For all $i\in \Z/m$ and for each arrow $a_i, \bar{a}_i$ of $Q$, we define a formal inverse by $a_i^{-1}$, $\bar{a}_i^{-1}$, respectively, and we let $\mathbf{s}(a_i)=\mathbf{t}(\bar{a}_i)=\mathbf{t}(a_i^{-1})=\mathbf{s}(\bar{a}_i^{-1})=i$ and $\mathbf{t}(a_i)=\mathbf{s}(\bar{a}_i)=\mathbf{s}(a_i^{-1})=\mathbf{t}(\bar{a}_i^{-1})=i+1$, i.e. $\mathbf{s}$ and $\mathbf{t}$ denotes the vertex where an arrow or formal inverse of an arrow starts and ends, respectively. By a {\it word} of length $n\geq 1$ we mean a sequence $w_1\cdots w_n$, where the $w_j$ is either an arrow or a formal inverse of an arrow, 
and where $\mathbf{s}(w_j)=\mathbf{t}(w_{j+1})$ for  $1\leq j \leq n-1$. We define $(w_1\cdots w_n)^{-1}={w_n}^{-1}\cdots {w_1}^{-1}$, $\mathbf{s}(w_1\cdots w_n)=\mathbf{s}
(w_n)$ and $\mathbf{t}(w_1\cdots w_n)=\mathbf{t}(w_1)$.  If $i\in \Z/m$ is a vertex of $Q$, we define an empty word $\1_i$ of length zero with $\mathbf{t}(\1_i)=i=\mathbf{s}
(\1_i)$ and $(\1_i)^{-1}=\1_i$. 

Note that in this article, we write words (consequently paths) from right to left. Denote by $\mathcal{W}$ the set of all words and let 
\[J=\{a_{i+1}a_i,\bar{a}_{i-1}\bar{a}_i, (\bar{a}_ia_i)^N,(a_{i-1}\bar{a}_{i-1})^N : i\in \Z/m\}.\] 
Let $\sim$ be the equivalence relation on $\mathcal{W}$ defined by  $w\sim w'$ if and only if $w=w'$ or $w^{-1}
=w'$. A \textit{string} is a representative $C$ of an equivalence class under the relation $\sim$,  where either  $C=\1_i$ for some vertex $i$ of $Q$, or $C=w_1\cdots w_n$ with $n
\geq 1$ and $w_j\not=w_{j+1}^{-1}$ for $1\leq j\leq  n-1$ and no sub-word of $C$ or its formal inverse belong to $J$. If $C$ is a string such that $\mathbf{s}(C)=\mathbf{t}(C)$, then we let 
$C^0=\1_{\mathbf{t}(C)}$. If $C=w_1\cdots w_n$ and $D=v_1\cdots v_m$ are strings of 
length $n,m\geq 1$, respectively, we say that the composition of $C$ and $D$ is defined provided that $w_1\cdots w_nv_1\cdots v_m$ is a string and write $CD=w_1\cdots 
w_nv_1\cdots v_m$. Observe that  $C\1_{\mathbf{s}(C)}\sim C$ and $\1_{\mathbf{t}(C)}C\sim C$. Moreover, if $C=w_1\cdots w_n$ is a string of length $n\geq 1$, then $C\sim w_1\cdots w_j\1_{\mathbf{t}(w_{j+1})}w_{j+1}\cdots w_n$ for 
all $1\leq j\leq n-1$. If $C=w_n\cdots w_1$ is a string of length $n\geq 1$, then there exists an indecomposable $\A_{m,N}$-module $M[C]$, called the {\it string 
module} corresponding to the string representative $C$, which can be  described as follows. There is an ordered $\k$-basis $\{z_0,z_1,\ldots,z_n\}$ of $M[C]$ such that the action of $
\A_{m,N}$ 
on $M[C]$ is given by the following representation $\varphi_C:\A_{m,N}\to \Mat(n+1,\k)$. Let $\text{\bf v}(j)=\mathbf{t}(w_{j+1})$ for $0\leq j\leq n-1$ and $\text{\bf v}(n)=\mathbf{s}(w_n)$. Then for each vertex $i\in \Z/m$, for each arrow $\gamma\in \{a_i, \bar{a}_i : i\in \Z/m\}$ in $Q$, and for all $0\leq j\leq n$ define
\begin{align*}
\varphi_C(i)(z_j)=\begin{cases}
z_j, &\text{ if $\text{\bf v}(j)=i$,}\\
0, &\text{ otherwise,}
\end{cases}
&&
\text{ and }
&&
\varphi_C(\gamma)(z_j)=\begin{cases}
z_{j-1}, &\text{ if $w_j=\gamma$,}\\
z_{j+1}, &\text{ if $w_{j+1}=\gamma^{-1}$,}\\
0, &\text{ otherwise.}
\end{cases}
\end{align*}

We call $\varphi_C$ the {\it canonical representation} and $\{z_0,z_1,\ldots,z_n\}$ a {\it canonical $\k$-basis} for $M[C]$ relative to the string representative $C$. Note that $M[C]
\cong M[C^{-1}]$.  If $C=\1_i$ with $i\in \Z/m$, 
then $M[C]=M[\1_i]$ is the simple $\A_{m,N}$-module corresponding to the vertex $i$.

\subsection{The components of the stable Auslander-Reiten quiver of $\A_{m,N}$ containg string $\A_{m,N}$-modules}\label{ape2}
In the following, we recall the description of the irreducible morphisms between string $\A_{m,N}$-modules provided in \cite{buri}. 

Assume $C=w_1w_2\cdots w_n$ with $n\geq 1$ is a string. We say that $C$ 
is \textit{directed} if all $w_j$ are arrows 
and  we say that $C$ is a {\it maximal directed string} if $C$ is directed and if for every arrow $\gamma$ in $Q$, $\gamma C\in J$.
Let $\mathcal{M}$ be the set of all maximal directed strings, i.e.,

\[\mathcal{M}=\{(a_i\bar{a}_i)^{N-1}a_i, (\bar{a}_{i-1}a_{i-1})^{N-1}\bar{a}_{i-1} : i\in \Z/m\}.\]

Let $C$ be a string for $\A_{m,N}$. We say that $C$ {\it starts on a peak} (resp., {\it starts in a deep}) provided that there is no arrow $\varphi$ in $Q$ such that $C\varphi$ (resp.,  $C
\varphi^{-1}$) is a string;  we also say that $C$ {\it ends on a peak} (resp., {\it ends in a deep}) provided that there is no arrow $\gamma$ in $Q$ such that  $\gamma^{-1}C$ (resp., $\gamma C$) is a string. If $C$ is a string 
not starting on a peak (resp., not starting in a deep), say $C\varphi$ (resp., $C\varphi^{-1}$) is a string for some arrow $\varphi$, then there is a unique directed string $D\in 
\mathcal{M}$ such that $C_h=C\varphi D^{-1}$  (resp., $C_c=C\varphi^{-1}D$) is a string. We say $C_h$ (resp., $C_c$) is obtained from $C$ by adding a {\it hook} 
(resp., a {\it co-hook}) {\it on the right side}.
Dually,  if $C$ is a string not ending on a peak (resp., not ending in a deep), say $\gamma^{-1} C$ (resp., $\gamma C$) is a string for some arrow $\gamma$ in $Q$, 
then there is a unique directed string $E\in \mathcal{M}$ such that $_hC=E\gamma^{-1} C$  (resp., $_cC=E^{-1}\gamma D$) is a string. We say $_hC$ (resp., $_cC$) 
is obtained from $C$ by adding a {\it hook} (resp., a {\it co-hook}) {\it on the left side}. By \cite{buri}, all irreducible morphisms between string modules are either canonical 
injections $M[C]\to M[C_h]$, $M[C]\to M[{_hC}]$, or canonical projections $M[C_c]\to M[C]$, $M[{_cC}]\to M[C]$. 

Let $S$ be a substring of $C$. We say that $S$ {\it starts on a peak} (resp., {\it starts in a deep}) in $C$ provided that there is no arrow $\varphi$ in $Q$ such that $S\varphi$ (resp.,  $S\varphi^{-1}$) is a substring of $C$;  we also say that $S$ {\it ends on a peak} (resp., {\it ends in a deep}) in $C$ provided that there is no arrow $\gamma$ in $Q$ such that  $\gamma^{-1}S$ (resp., $\gamma S$) is a substring in $C$.

We say that a string $\A_{m,N}$-module $M[C]$ is of {\it minimal string length} if the string representative $C$ is not in $\mathcal{M}$ and it cannot be obtained by adding a hook or a co-hook to a proper substring. If $M[C]$ is a string module of minimal string length, then $M[C]$ belongs to a component $\mathfrak{C}$ of $\Gamma_s(\A_{m,N})$ of type $\mathbb{ZA}_\infty^
\infty$, i.e., near $M[C]$ the component $\mathfrak{C}$ looks as in Figure \ref{component}.

\begin{figure}[ht]
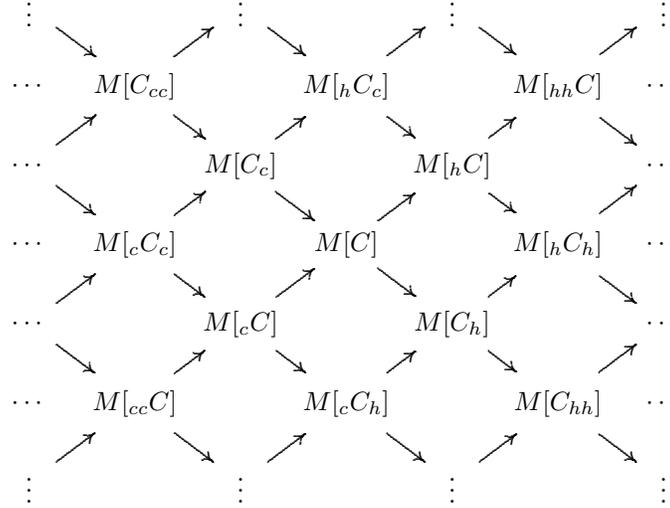

$
\begindc{\commdiag}[100]
\obj(-12,10)[a]{$\vdots$}
\obj(-4,10)[1]{$\vdots$}
\obj(4,10)[2]{$\vdots$}
\obj(12,10)[b]{$\vdots$}
\obj(-12,7)[3]{$\cdots$}
\obj(-8,7)[4]{$M[C_{cc}]$}
\obj(0,7)[5]{$M[_hC_{c}]$}
\obj(8,7)[6]{$M[_{hh}C]$}
\obj(12,7)[7]{$\cdots$}

\obj(-12,4)[8]{$\cdots$}
\obj(-4,4)[9]{$M[C_{c}]$}
\obj(4,4)[10]{$M[_hC]$}
\obj(12,4)[11]{$\cdots$}

\obj(-12,1)[12]{$\cdots$}
\obj(-8,1)[13]{$M[_cC_{c}]$}

\obj(0,1)[14]{$M[C]$}
\obj(8,1)[15]{$M[_hC_h]$}
\obj(12,1)[16]{$\cdots$}

\obj(-12,-2)[17]{$\cdots$}
\obj(-4,-2)[18]{$M[_cC]$}
\obj(4,-2)[19]{$M[C_h]$}
\obj(12,-2)[20]{$\cdots$}

\obj(-12,-5)[21]{$\cdots$}
\obj(-8,-5)[22]{$M[_{cc}C]$}
\obj(0,-5)[23]{$M[_cC_h]$}
\obj(8,-5)[24]{$M[C_{hh}]$}
\obj(12,-5)[25]{$\cdots$}

\obj(-12,-8)[26]{$\vdots$}
\obj(-4,-8)[27]{$\vdots$}
\obj(4,-8)[28]{$\vdots$}
\obj(12,-8)[29]{$\vdots$}

\mor{a}{4}{}
\mor{4}{1}{}
\mor{1}{5}{}
\mor{5}{2}{}
\mor{2}{6}{}
\mor{6}{b}{}
\mor{8}{4}{}
\mor{4}{9}{}
\mor{9}{5}{}
\mor{5}{10}{}
\mor{10}{6}{}
\mor{6}{11}{}
\mor{8}{13}{}
\mor{13}{9}{}
\mor{9}{14}{}
\mor{14}{10}{}
\mor{10}{15}{}
\mor{15}{11}{}
\mor{17}{13}{}
\mor{13}{18}{}
\mor{18}{14}{}
\mor{14}{19}{}
\mor{19}{15}{}
\mor{15}{20}{}
\mor{17}{22}{}
\mor{22}{18}{}
\mor{18}{23}{}
\mor{23}{19}{}
\mor{19}{24}{}
\mor{24}{20}{}
\mor{26}{22}{}
\mor{22}{27}{}
\mor{27}{23}{}
\mor{23}{28}{}
\mor{28}{24}{}
\mor{24}{29}{}
\enddc
$
\caption{The stable Auslander-Reiten component near $M[C]$.}\label{component}
\end{figure}

\subsection{Homomorphisms between string $\A_{m,N}$-modules}\label{appendix3}
Let $S$ and $T$ be strings for $\A_{m,N}$ and let $M[S]$ and $M[T]$ their respective string $\A_{m,N}$-modules with respective canonical $\k$-basis $\{x_u\}_{u=0}^{l_1}$ and $\{y_v\}_{v=0}^{l_2}$. Suppose that $C$ is a substring of both $S$ and $T$ such that the following conditions (i) and (ii) are satisfied.
\begin{enumerate}
\item $S\sim S'CS''$, with ($S'$ of length zero or $S'=\hat{S}'\xi_1$) and ($S''$ of length zero or $S''=\xi_2^{-1}\hat{S}''$), where $S'$, $\hat{S}'$, $S''$, $\hat{S}''$ are strings and $\xi_1$, $\xi_2$ are arrows in $Q$; and 
\item $T\sim T'CT''$, with ($T'$ of length zero or $T'=\hat{T}'\zeta_1^{-1}$) and ($T''$ of length zero or $T''=\zeta_2\hat{T}''$), where $T'$, $\hat{T}'$, $T''$, $\hat{T}''$ are strings and $\zeta_1$, $\zeta_2$ are arrows in $Q$.
\end{enumerate}
Then by \cite{krause} there exists a composition of  $\A_{m,N}$-module homomorphisms
\begin{equation}\label{canhom}
\sigma_C:M[S]\surjection M[C]\injection M[T], 
\end{equation}
and which sends each element of $\{x_u\}_{0\leq u\leq l_1-1}$ either to zero or to an element of $\{y_v\}_{0\leq v\leq l_2-1}$, according to the relative position of $C$ in $S$ and $T$, respectively. Suppose e.g. that $S=w_1\cdots w_{l_1}$, $T=\tilde{w}_1\cdots \tilde{w}_{l_2}$ and $C=w_{j+1}w_{j+2}\cdots w_{j+l_3}=\tilde{w}_{k+l_3}^{-1}\tilde{w}_{k+l_3-1}^{-1}\cdots \tilde{w}_{k+1}^{-1}$, then   
\begin{align*}
\sigma_C(x_{j+t})=y_{k+l_3-t} &&\text{ for $0\leq t\leq l_3$},  &&\text{ and } &&\sigma_C(x_u)=0 &\text{ for all other $u$}. 
\end{align*}

We call $\sigma_C$ a {\it canonical homomorphism from $M[S]$ to $M[T]$}. Note that there may be several choices for $S'$, $S''$ (resp. $T'$, $T''$) in (i) (resp. (ii)). In other words, there may be several $\k$-linearly independent canonical homomorphisms factoring through $M[C]$. By \cite{krause},  every $\A_{m,N}$-module homomorphism $\sigma:M[S]\to M[T]$ can be written as a unique $\k$-linear combination of canonical homomorphisms which factor through string modules corresponding to strings $C$ satisfying (i) and (ii). In particular, if $S=T$, then every $\A_{m,N}$-module endomorphism of $M[S]$ can be written as a unique $\k$-linear combination of the identity homomorphism and of canonical endomorphisms which factor through string $\A_{m,N}$-modules $M[C]$ for suitable choices of $C$ satisfying (i) and (ii). 
We call $\sigma_C$ a {\it canonical homomorphism} from $M[S]$ to $M[T]$ that factors through $M[C]$. It follows from \cite{krause} that each $\A_{m,N}$-module homomorphism 
from $M[S]$ to $M[T]$ can be written uniquely as a $\k$-linear combination of canonical $\A_{m,N}$-module homomorphisms as in (\ref{canhom}). In particular, if $M[S]=M[T]$ 
then the canonical endomorphisms of $M[S]$ generate $\End_{\A_{m,N}}(M[S])$.

\bibliographystyle{amsplain}
\bibliography{UniversalDeformationSelfinjective(Rev)}   

\providecommand{\bysame}{\leavevmode\hbox to3em{\hrulefill}\thinspace}
\providecommand{\MR}{\relax\ifhmode\unskip\space\fi MR }
\providecommand{\MRhref}[2]{%
  \href{http://www.ams.org/mathscinet-getitem?mr=#1}{#2}
}
\providecommand{\href}[2]{#2}
\begin{thebibliography}{10}

\bibitem{assem3}
I.~Assem, I.~Simson, and A.~Skowro\'nski, \emph{Elements of the
  {R}epresentation {T}heory of {A}ssociative {A}lgebras. {V}olume {I}
  {T}echniques of {R}epresentation {T}heory}, London Mathematical Society
  Student Texts, no.~65, Cambridge University Press, 2006.

\bibitem{auslander4}
M.~Auslander and M.~Brigder, \emph{Stable {M}odule {T}heory}, Memoirs of the
  American Mathematical Society, no.~94, American Mathematical Society, 1969.

\bibitem{auslander}
M.~Auslander, I.~Reiten, and S.~Smal{\o}, \emph{Representation {T}heory of
  {A}rtin {A}lgebras}, Cambridge Studies in Advanced Mathematics, no.~36,
  Cambridge University Press, 1995.

\bibitem{benson}
D.~J. Benson, \emph{Representations and {C}ohomology {I}: {B}asic
  representation theory of finite groups and associative algebras}, Cambridge
  Studies in Advanced Mathematics 30, Cambridge University Press, 1991.

\bibitem{bleher1}
F.~M. Bleher, \emph{Universal deformation rings of dihedral defect groups},
  Trans. Amer. Math. Soc \textbf{361} (2009), 3661--3705.

\bibitem{bleher7}
\bysame, \emph{Universal deformation rings and generalized quaternion defect
  groups}, Adv. Math. \textbf{225} (2010), 1499--1522.

\bibitem{bleher2}
F.~M. Bleher and T.~Chinburg, \emph{Universal deformation rings and cyclic
  blocks}, Math. Ann. \textbf{318} (2000), 805--836.

\bibitem{bleher3}
\bysame, \emph{Universal deformation rings need not be complete intersections},
  Math. Ann. \textbf{337} (2007), 739--767.

\bibitem{bleher4}
F.~M. Bleher, T.~Chinburg, and B.~de~Smith, \emph{Inverse problems for
  deformation rings}, Trans. Amer. Math. Soc \textbf{365} (2013), no.~11,
  6149--6165.

\bibitem{bleher5}
F.~M. Bleher and J.~B. Froelich, \emph{Universal deformation rings for the
  symmetric group ${S}_5$ and one of its double covers}, J. Pure Appl. Algebra
  \textbf{215} (2011), 523--530.

\bibitem{bleher6}
F.~M. Bleher and G.~Llosent, \emph{Universal deformation rings for the
  symmetric group ${S}_4$}, Algebr. Represent. Theory \textbf{13} (2010),
  255--270.

\bibitem{bleher9}
F.~M. Bleher and S.~N. Talbott, \emph{Universal deformation rings of modules
  for algebras of dihedral type of polynomial growth}, Algebr. Represent.
  Theory \textbf{17} (2014), 289--303.

\bibitem{blehervelez2}
F.~M. Bleher and J.~A. V{\'e}lez-Marulanda, \emph{Deformations of complexes for
  finite dimensional algebras}, Accepted in J. Algebra. Available in
  \href{http://arxiv.org/abs/1511.08081}{http://arxiv.org/abs/1511.08081}.

\bibitem{blehervelez}
\bysame, \emph{Universal deformation rings of modules over {F}robenius
  algebras}, J. Algebra \textbf{367} (2012), 176--202.

\bibitem{broue}
M.~Brou{\'e}, \emph{Equivalences of blocks of group algebras}, Finite
  {D}imensional {A}lgebras and {R}elated {T}opics ({V}. Dlab and {L}.~{L}.
  Scott, eds.), {NATO} {ASI} {S}eries, no. 424, Springer Netherlands, 1994,
  pp.~1--26.

\bibitem{buri}
M.~C.~R. Butler and C.~M. Ringel, \emph{{A}uslander-{R}eiten sequences with few
  middle terms and applications to string algebras}, Comm. Algebra \textbf{15}
  (1987), 145--179.

\bibitem{curtis}
C.~W. Curtis and I.~Reiner, \emph{Methods of {R}epresentation {T}heory with
  {A}pplications to {F}inite {G}roups and {O}rders}, vol.~I, John Wiley and
  Sons, New York, 1981.

\bibitem{erdmann}
K.~Erdmann, \emph{{B}locks of {T}ame {R}epresentation {T}ype and {R}elated
  {A}lgebras}, Lectures Notes in Mathematics, no. 1428, Springer-Verlag, 1990.

\bibitem{erdmann2}
\bysame, \emph{On {H}ochschild cohomology for selfinjective special biserial
  algebras}, Algebras, {Q}uivers and {R}epresentations: The {A}bel {S}ymposium
  2011 (A.~B. Buan, I.~{R}eiten, and {\O}.~{S}olberg, eds.), Abel {S}ymposia,
  no.~8, Springer, 2013, pp.~79--94.

\bibitem{erdmann3}
K.~Erdmann and A.~Skowro\'nski, \emph{On {A}uslander-{R}eiten components of
  blocks and self-injective biserial algebras}, Trans. Amer. Math. Soc.
  \textbf{330} (1992), no.~1, 165--189.

\bibitem{holm}
T.~Holm, \emph{Derived equivalence classification of algebras of dihedral,
  semidihedral, and quaternion type}, J. Algebra \textbf{211} (1999), 159--205.

\bibitem{ile}
R.~Ile, \emph{Change of rings in deformation theory of modules}, Trans. Amer.
  Math. Soc \textbf{356} (2004), 4873--4896.

\bibitem{krause}
H.~Krause, \emph{Maps between tree and band modules}, J. Algebra \textbf{137}
  (1991), 186--194.

\bibitem{mazur}
B.~Mazur, \emph{An introduction to the deformation theory of {G}alois
  representations}, Modular{ F}orms and {F}ermat's {L}ast {T}heorem ({G}.
  {C}ornell, {J}.~{H}. {S}ilverman, and {G.} {S}tevens, eds.), Springer-Verlag,
  Boston, MA, 1997, pp.~243--311.

\bibitem{snashall2}
A.~Parker and N.~Snashall, \emph{A family of {K}oszul self-injective algebras
  with finite {H}ochschild cohomology}, J. Pure Appl. Algebra \textbf{216}
  (2012), 1245--1252.

\bibitem{sch}
M.~Schlessinger, \emph{Functors of {A}rtin rings}, Trans. Amer. Math. Soc.
  \textbf{130} (1968), 208--222.

\bibitem{schroll}
S.~Schroll, \emph{Trivial extensions of gentle algebras and {B}rauer graph
  algebras}, J. Algebra \textbf{444} (2015), 183--200.

\bibitem{skow}
A.~Skowro{\'n}ski and K.~Yamagata, \emph{Selfinjective algebras of quasitilted
  type}, Trends in {R}epresentation {T}heory of {A}lgebras and {R}elated
  {T}opics ({A.} {S}kowro{\'n}ski, ed.), EMS Series of Congress Reports,
  European Mathematical Society, 2008, pp.~639--708.

\bibitem{snashall1}
N.~Snashall and R.~Taillefer, \emph{The {H}ochschild cohomology ring of a class
  of special biserial algebras}, J. Algebra Appl. \textbf{9} (2010), no.~1,
  73--122.

\bibitem{velez}
J.~A. V{\'e}lez-Marulanda, \emph{Universal deformation rings of strings modules
  over a certain symmetric special biserial algebra}, Beitr. Algebra Geom.
  \textbf{56} (2015), no.~1, 129--146.

\bibitem{wald}
B.~Wald and J.~{W}aschb{\"u}sch, \emph{Tame biserial algebras}, J. Algebra
  \textbf{95} (1985), 480--500.

\bibitem{yau}
D.~Yau, \emph{Deformation theory of modules}, Comm. Algebra \textbf{33} (2005),
  2351--2359.

\end{thebibliography}



\end{document}